\documentclass[12pt,a4wide]{article}

\usepackage[a4paper,margin=2.5cm]{geometry}

\usepackage[usenames,dvipsnames]{xcolor}
\usepackage{amssymb,mathdots}
\usepackage{mathrsfs}
\usepackage{amsmath}
\usepackage{latexsym}
\usepackage{pstricks,pst-plot,pst-node,pst-text,pst-3d}
\usepackage{color,graphicx}
\usepackage{ulem}
\usepackage{relsize}
\definecolor{musgo}{RGB}{0,100,0}

\makeatletter
\newcommand{\Vast}{\bBigg@{5}}
\makeatother

\newtheorem{theorem}{Theorem}[section]
\newtheorem{lemma}[theorem]{Lemma}

\newtheorem{definition}[theorem]{Definition}
\newtheorem{example}[theorem]{Example}
\newtheorem{remark}[theorem]{Remark}
\newenvironment{proof}[1][Proof]{\begin{trivlist}
\item[\hskip \labelsep {\bfseries #1}]}{\end{trivlist}}

\def\zede{\textit{\texttt{z}}}

\title{Moving Frames and Noether's Conservation\\ Laws -- the General Case}
\author{T\^ania M. N. \ Gon\c calves $^{a,} $\footnote{Author supported by PNPD/CAPES -- Programa Nacional de P\'{o}s Doutorado, Brazil.} \and Elizabeth L.\ Mansfield $^{b,}$\footnote{{Author supported by EPSRC, UK, Grant EP/H024018/1.}}}
\date{}

\begin{document}
\maketitle

\begin{center}
\footnotesize{$^a$ Departamento de Matem\'{a}tica, Universidade Federal de S\~{a}o Carlos\\ Rod. Washington Lu\'{i}z, Km 235 - C.P. 676 - 13565-905 S\~{a}o Carlos, SP - Brazil \\$^b$ School of Mathematics, Statistics and Actuarial Science, University of Kent\\ Canterbury, CT2 7NF, United Kingdom\\Emails: \texttt{tmng@kentforlife.net},$\quad$ \texttt{E.L.Mansfield@kent.ac.uk}}
\end{center}

\begin{abstract}
In recent works \cite{Mansfield,GoncalvesMansfield}, the authors considered various Lagrangians, which are invariant under a Lie group action, in the case where the independent variables are themselves invariant. Using a moving frame for the Lie group action, they showed how to obtain the invariantized Euler-Lagrange equations and the space of conservation laws in terms of vectors of invariants and the adjoint representation of a moving frame.

In this paper, we show how these calculations extend to the general case where the independent variables may participate in the action. We take for our main expository example the standard linear action of SL(2) on the two independent variables. This choice is motivated by applications to variational fluid problems which conserve potential vorticity. We also give the results for Lagrangians invariant under the standard linear action of SL(3) on the three independent variables.
\end{abstract}

\section{Introduction}
Noether's First Theorem states that for systems coming from a variational principle, conservation laws may be obtained from Lie group actions which leave the Lagrangian invariant.

Recently in \cite{Mansfield,GoncalvesMansfield}, for the case where the invariant Lagrangians may be parametrized so that the independent variables are each invariant under the group action, the authors were able to calculate the invariantized Euler-Lagrange system in terms of the standard Euler operator and a `syzygy' operator specific to the action. Furthermore, they obtained the linear space of conservation laws in terms of vectors of invariants and the adjoint representation of a moving frame for the Lie group action. This new structure for the conservation laws allowed the calculations for the extremals to be reduced and given in the original variables, once the Euler-Lagrange system was solved for the invariants. These results were presented in \cite{GoncalvesMansfield} for all three inequivalent SL(2) actions in the complex plane and in \cite{GoncalvesMansfieldII} for the standard SE(3) action.

In this paper, we show that the results presented in \cite{GoncalvesMansfield} can be extended to cases where the independent variables are not invariant under the group action, which is the case for many physically important models. In Table \ref{tablegpact} we list some conservation laws arising from group actions on the base space. We take as our main expository example the standard linear action of SL(2) on the two independent variables due to its importance in variational problems which conserve potential vorticity. Indeed in  \cite{Hydon,BridgesHydonReich}, Bridges et al. give rigorous connection between particle relabelling, symplecticity and conservation of potential vorticity; they show that conservation of potential vorticity  is a differential consequence of a 1-form quasi-conservation law, which is obtained from rewriting the shallow water equations as a multisymplectic system. Here, we will show that conservation of potential vorticity is a differential consequence of Noether's conservation laws for the SL(2) action.

\begin{table}[h]
\begin{center}
\begin{tabular}{|ll|}
\hline
Group action & Conservation law\\
\hline
Time translation & Energy\\
Space translation & Linear momentum\\
Space rotation & Angular momentum\\
Area preserving diffeomorphism & Potential vorticity\\
\hline
\end{tabular}
\end{center}
\footnotesize{\textbf{\caption{Conservation laws arising from group actions on the base space.}}}\label{tablegpact}
\end{table}

In Section \ref{movingframesICV}, we start by giving some background on moving frames, differential invariants, invariant differentiation operators, and invariant forms. We then move on to the invariant calculus of variations; we show in this section how the invariantized Euler-Lagrange equations are obtained in a way similar to that of the Euler-Lagrange equations in the original variables. 

In Section \ref{structureNCL}, we show how the variational symmetry group acts on Noether's conservation laws and demonstrate the mathematical structure of Noether's conservation laws for invariant Lagrangians with independent variables that are not invariant under the group action. The conservation laws presented in this section are a generalization of the ones obtained in \cite{GoncalvesMansfield}; they differ by the product of a matrix which represents the group action on the $(p-1)$-forms. In the particular case of a variational problem with invariant independent variables, this matrix corresponds to the identity matrix. We end this section with the calculation of conservation laws associated to the Monge-Amp\`{e}re equation.

In Section \ref{twoexamples}, we compute the new version of Noether's conservation laws which are associated to two three-dimensional invariant variational problems -- the shallow water equations and Lagrangians invariant under the linear SL(3) action on the base space. We conclude with some remarks about the form of the Euler-Lagrange equations in terms of the conservation laws, that follow as a consequence of our main result.

\subsection{Motivating example}
Consider the following $\mathrm{SL(2)}$ group action on the $(x,u(x))$-plane,
\begin{equation}\label{firstgroupaction}
g\cdot x =\widetilde{x}=\dfrac{ax+b}{cx+d},\qquad g\cdot u=\widetilde{u}=u,
\end{equation}
where $ad-bc=1$. The following expression 
$$\sigma=\dfrac{u_{xxx}}{u_x^3}-\dfrac{3}{2}\dfrac{u_{xx}^2}{u_x^4},$$
is the lowest order differential invariant, where a differential invariant is an invariant for the prolonged group action of a Lie group on a jet-space. All differential invariants for the group action (\ref{firstgroupaction}) are functions of $\sigma$ and the invariant differentiation operator $\mathcal{D}_x=\frac{1}{u_x}\frac{D}{Dx}$.

Under this group action, the one-dimensional variational problem 
$$\int \dfrac{(2u_{xxx}u_x-3u_{xx}^2)^2}{4u_x^7}\,\mathrm{d}x=\int \sigma^2u_x\mathrm{d}x$$
has SL(2) as a variational symmetry group. Using the formula for Noether's conservation laws, as formulated in $\S 5.4$, Proposition 5.98 of \cite{Olver}, we obtain a system of conservation laws which can be written in matrix form as $A(x,u_x,u_{xx})\boldsymbol{\upsilon}(I)=\mathbf{c}$, where $\boldsymbol{\upsilon}(I)$ is a vector of invariants, and $\mathbf{c}$ are the constants of integration; more precisely, we have
\begin{equation}\label{onedimvarprob}
\left(\begin{array}{ccc} 
\dfrac{xu_{xx}+u_x}{u_x} & 2 x u_x & -\dfrac{u_{xx}(xu_{xx}+2u_x)}{2u_x^3}\\
\dfrac{u_{xx}}{2u_x} & u_x & -\dfrac{u_{xx}^2}{4u_x^3}\\
-\dfrac{x(xu_{xx}+2u_x)}{2u_x} & -x^2u_x & \dfrac{(xu_{xx}+2u_x)^2}{4u_x^3}
\end{array}\right)\left(\begin{array}{c}
-4\mathcal{D}_x\sigma\\
-2\sigma^2+2\mathcal{D}_x^2\sigma\\
-4\sigma\end{array}\right)=\left(\begin{array}{c}c_1\\c_2\\c_3\end{array}\right),
\end{equation}
where this defines $A$ and $\boldsymbol{\upsilon}(I)$.

The Euler-Lagrange equation for this variational problem is $-2\mathcal{D}_x^3\sigma+6\sigma\mathcal{D}_x\sigma=0$,
i.e.
$$(-\mathcal{D}_x^3+2\mathcal{D}_x\sigma+2\sigma\mathcal{D}_x)\mathsf{E}^\sigma(L)+\mathcal{D}_x\left(-L\right)=0,$$
where $\mathsf{E}^\sigma$ is the Euler operator with respect to $\sigma$. This invariantized Euler-Lagrange equation agrees with the invariant form given in Kogan and Olver \cite{KoganOlver},
\begin{equation}\label{IELKO}
\mathcal{A}^\ast\mathcal{E}(L)-\mathcal{B}^\ast\mathcal{H}(L)=0,
\end{equation}
where $\mathcal{E}(L)$ is the invariantized Eulerian, $\mathcal{H}(L)$ a suitable invariantized Hamiltonian, and $\mathcal{A}^\ast,\,\mathcal{B}^\ast$, which are named \textit{Eulerian} and \textit{Hamiltonian operators}, respectively, are invariant differential operators.
  
Once one has solved the Euler-Lagrange equation for $\sigma$ and substituted $\sigma$ in the system of conservation laws (\ref{onedimvarprob}), one obtains three equations for $u_x$ and $u_{xx}$ as functions of $x$. Combining and simplifying these yields
\begin{eqnarray}\label{solution1}
u_x(c_1x-c_2x^2+c_3)+4\sigma=0.
\end{eqnarray}
Equation (\ref{solution1}) can be solved for $u$, once the solution to $\sigma$ is known. As shown in \cite{GoncalvesMansfield}, for one-dimensional invariant variational problems, it may be possible for the the system of conservation laws to be used to solve for the extremals, provided the Adjoint representation is non trivial.

The matrix $A$ defined in (\ref{onedimvarprob}) is \textit{equivariant}, in other words, letting the group act on its components, then one can verify that the group action factors out; more precisely,
$$A(\widetilde{x},\widetilde{u_x},\widetilde{u_{xx}})=R(a,b,c)A(x,u_x,u_{xx}),$$
where
$$R(a,b,c)=\left( \begin{array}{ccc}
ad+bc & 2bd & -2ac\\
cd & d^2 & -c^2\\
-ab & -b^2 &a^2
\end{array}\right), \qquad d=\dfrac{1+bc}{a}.$$
The matrix $R(a,b,c)$ is a representation of SL(2); the group product in parameter space is given by
$$(a,b,c)\cdot (\alpha,\beta,\gamma)=(a\alpha+b\gamma,a\beta+b\delta,c\alpha+d\gamma),\qquad d=\dfrac{1+bc}{a},\qquad \delta=\dfrac{1+\beta\gamma}{\alpha},$$
and it is easily checked that 
$$R(a,b,c)\cdot R(\alpha,\beta,\gamma)=R((a,b,c)\cdot (\alpha,\beta,\gamma)).$$
This representation is the well-known adjoint representation, see \S 3.3 of \cite{Mansfield}. In fact, the map $A$ is a \textit{moving frame}, i.e. an equivariant map from the space $M$ on which the Lie group $G$ acts, in this case, the relevant jet bundle, to the group itself.

At first glance the structure of the conservation laws, for invariant variational problems whose independent variables are also invariant, seems to be identical to the one where the independent variables participate in the action. However, as will be shown later, some of the terms in the vector of invariants come from the Lie derivative of the invariant volume form with respect to the variation parameter; the difference becomes more visible in higher dimensional variational problems, as the conservation laws will also pick up an extra matrix term. 

\section{Moving frames and invariant calculus of variations}\label{movingframesICV}
In this section, we will introduce notions and concepts needed to understand our results, namely, moving frames as formulated by Fels and Olver \cite{FelsOlver,FelsOlverII} in the context of differential algebra, differential invariants of a group action, invariant differential operators, invariant forms and invariant calculus of variations. For further details on these topics see Fels and Olver \cite{FelsOlver,FelsOlverII}, and Mansfield \cite{Mansfield}. Also, a different approach to invariant calculus of variations can be found in Kogan and Olver \cite{KoganOlver} .

We will start by defining what a moving frame is and then use it to obtain the differential invariants, the invariant differential operators and the invariant differential forms. Then we will proceed to the topic of invariant calculus of variations, where we explain how the invariantized Euler-Lagrange equations are calculated. In the process of obtaining these, a collection of boundary terms are picked up; as will be seen in Section \ref{structureNCL}, these will yield part of the new structured version of Noether's conservation laws in terms of invariants and a moving frame. 

\subsection{Moving frames and differential invariants}
A smooth group acting on a smooth space induces an action on the set of its smooth curves and surface elements and on their higher order derivatives  in the relevant jet bundle. These curves and surfaces are known as the \textit{prolonged curves} and \textit{surfaces}. In this paper, the set $M$ on which the group $G$ acts is the set of these prolonged curves and surfaces.

Let $X$ be the set of independent variables with coordinates $\mathbf{x}=(x_1,...,x_p)$ and $U$ the set of dependent variables with coordinates $\mathbf{u}=(u^1,...,u^q)$. We will represent the derivatives of $u^\alpha$ with a multi-index notation, e.g.
$$u^\alpha_\mathrm{K}=\frac{\partial^{|K|} u^\alpha}{\partial x_{k_1}\cdots\partial x_{k_m}},$$
where $\mathrm{K}=(k_1,....,k_m)$ is an unordered $m$-tuple of integers, where the entries $1\le k_\ell\le p$ represent the derivatives with respect to $x_{k_\ell}$; its order is denoted by $|\mathrm{K}|=m$. Consequently, we will represent the coordinates of $M=J^n(X\times U)$ as
$$\zede=(x_1,...,x_p,u^1,...,u^q,u^1_1,...).$$
Furthermore, the operator $\partial/ \partial x_i$ extends to the \textit{total differentiation operator}
$$D_i=\frac{\mathrm{d}}{\mathrm{d}x_i}=\frac{\partial}{\partial x_i}+\sum_{\alpha=1}^q\sum_{\mathrm{K}} u^\alpha_{\mathrm{K}i}\frac{\partial}{\partial u^\alpha_\mathrm{K}},$$
where $D_i$ maps $J^n$ into $J^{n+1}$.

A group action of $G$ on $M$ is a map 
$$G\times M  \rightarrow  M, \qquad (g,\zede) \mapsto  g\cdot\zede,$$
which satisfies either $g\cdot (h\cdot\zede)=(gh)\cdot\zede$, called a \textit{left action}, or $g\cdot (h\cdot\zede)=(hg)\cdot\zede$, called a \textit{right action}. To ease exposition, we will denote at times $g\cdot \zede$ as $\widetilde{\zede}$.

Suppose that $G$ is a Lie group acting smoothly on $M$ and that its action is free and regular in some domain $\mathcal{U}\subset M$. This implies that
\begin{itemize}
\item[-] the group orbits all have the same dimension and foliate $\mathcal{U}$,
\item[-] the existence of a surface $\mathcal{K}$ that intersects these orbits transversally, and for which the intersection with a given group orbit is a single point. This surface $\mathcal{K}$ is known as \textit{cross section}, and 
\item[-] if $\mathcal{O}(\zede)$ is an orbit through $\zede$, then the element $h \in G$ which maps $\zede$ to $\{k\}=\mathcal{O}(\zede)\cap \mathcal{K}$ is unique.
\end{itemize}

Under these conditions we can define an equivariant map $\rho: \mathcal{U}\rightarrow M$
as the map that sends an element $\zede \in \mathcal{U}$ to the unique element $\rho(\zede)\in G$ which satisfies 
$$\rho(\zede)\cdot \zede=k.$$
The map $\rho$ is called the \textit{right moving frame} relative to the cross section $\mathcal{K}$. 

To obtain the right moving frame, in a first instance, we define the cross section $\mathcal{K}$ as the locus of the set of equations $\psi_i(\zede)=0$, for $i=1,...,r$, where $r$ is the dimension of $G$. Then solving the set of equations, known as the \textit{normalization equations},
$$\psi_i(\widetilde{z})=\psi_i(g\cdot \zede)=0,\qquad i=1,...,r,$$
for the $r$ parameters describing $G$ yields the frame in parametric form.

\begin{example}\label{linearsl2}
Consider the linear SL(2) action on the space $(x,y,u(x,y))$ as follows
\begin{equation}\label{sl2onindependentvars}
\left(\begin{array} {c}\widetilde{x} \\ \widetilde{y}\end{array}\right)=\left(\begin{array}{cc} a & b\\ c & d\end{array}\right)\left(\begin{array}{c}x\\ y\end{array}\right), \qquad ad-bc=1,\qquad \widetilde{u}=u.
\end{equation}

The prolonged actions on $u_x$ and $u_y$ are given explicitly by $g\cdot u_x=\widetilde{u_x}=\widetilde{D_x}\widetilde{u}$ and $g\cdot u_y=\widetilde{u_y}=\widetilde{D_y}\widetilde{u}$, respectively.

The \textnormal{transformed total differentiation operators} $\widetilde{D_i}$ are defined by
\begin{equation}\label{transformedtotop}
\widetilde{D_i}=\frac{\mathrm{d}}{\mathrm{d}\widetilde{x_i}}=\sum_{k=1}^p((\mathrm{d}\widetilde{\mathbf{x}}/\mathrm{d}\mathbf{x})^{-T})_{ik}D_k,
\end{equation}
where $\mathrm{d}\widetilde{\mathbf{x}}/\mathrm{d}\mathbf{x}$ is the Jacobian matrix. So,
$$\widetilde{u_x}=du_x-cu_y,\qquad\widetilde{u_y}=-b u_x+a u_y.$$
Taking $M$ to be the space with coordinates $(x,y,u,u_x,u_y,u_{xx},u_{xy},u_{yy},...)$, then the action is locally free near the identity of SL(2) and regular away from the coordinate planes $x=0$ and the locus of $x u_x+yu_y=0$. In this domain, we may take the normalization equations to be $\widetilde{x}=1$, $\widetilde{y}=0$ and $\widetilde{u_y}=0$, and thus obtain
\begin{equation}\label{framesl2}
a=\frac{u_x}{xu_x+yu_y},\quad b=\frac{u_y}{xu_x+yu_y},\quad and \quad c=-y,
\end{equation}
as the frame in parametric form. 
\end{example}

\begin{theorem}
Let $\rho$ be a right moving frame, then the quantity $I(\zede)=\rho(\zede)\cdot\zede$ is an invariant of the group action (see \cite{FelsOlver}).
\end{theorem}

If $\zede$ is given in coordinates, and the normalization equations are $\widetilde{z}_i=c_i$, for $i=1,...,r$, then
$$\rho(\zede)\cdot \zede=(c_1,...,c_r,I(z_{r+1}),...,I(z_{n})),$$
where 
$$I(z_k)=g\cdot z_k|_{g=\rho(\zede)}, \quad \mathrm{for}\quad k=r+1,...,n.$$
Thus, we denote the invariantized jet bundle coordinates as
$$J^{\,i}=I(x_i)=\widetilde{x_i}|_{g=\rho(\zede)},\quad I^\alpha_\mathrm{K}=I(u^\alpha_\mathrm{K})=\widetilde{u^\alpha_\mathrm{K}}|_{g=\rho(\zede)}.$$
These are also known as the \textit{normalized differential invariants}. This follows the notation in \cite{FelsOlverII}. Other notations appearing in the literature are $\iota(\zede)$ and $\bar{\iota}\zede$.

\vspace{0.4cm}
\noindent\textbf{Example \ref{linearsl2} (cont.)}
\textit{The normalized differential invariants up to order two are as follows} 
\begin{align}\nonumber
& \kern -20pt g\cdot \zede=(\widetilde{x},\widetilde{y},\widetilde{u},\widetilde{u_x},\widetilde{u_y},\widetilde{u_{xx}},\widetilde{u_{xy}},\widetilde{u_{yy}})|_{g=\rho(\zede)}\\ \nonumber
&= (J^x,J^y,I^u,I^u_1,I^u_2,I^u_{11},I^u_{12},I^u_{22})\\ \nonumber
&=\Big(1,0,u,xu_x+yu_y,0, x^2u_{xx}+2xyu_{xy}+y^2u_{yy},\\\label{basicExListofDiffInvs}
&\qquad\displaystyle{\frac{xu_xu_{xy}-yu_yu_{xy}+yu_xu_{yy}-xu_yu_{xx}}{xu_x+yu_y},\frac{u_x^2u_{yy}-2u_xu_yu_{xy}+u_y^2u_{xx}}{(xu_x+yu_y)^2}\Big)}.
\end{align}
\textit{The first, second and fifth components correspond to the normalization equations and are known as the \textnormal{phantom invariants}. We will see that the third and eighth components, $u=I(u)$ and $I(u_{yy})$ respectively,  are the generating invariants and one can obtain all the higher order invariants in terms of them and their derivatives (we refer to Chapter 5 of \cite{Mansfield}  for a discussion of the relevant results that allow such claims to be proved). }

\subsection{Invariant differential operators and differential forms}
The \textit{invariant differential operators}  are calculated in a similar way to that of the normalized differential invariants. We obtain them by evaluating the transformed total differentiation operators at the frame, in other words,
$$\mathcal{D}_i=\widetilde{D_i}|_{g=\rho(\zede)},$$
where $\widetilde{D_i}$ are as defined in (\ref{transformedtotop}). These invariant differentiation operators map differential invariants to differential invariants. 

We know that $\partial u^\alpha_\mathrm{K}/\partial x_i=u^\alpha_{\mathrm{K}i}$, but the same is not true for their invariantized counterparts;  in general
$$\mathcal{D}_i I^\alpha_\mathrm{K}\neq I^\alpha_{\mathrm{K}i}.$$
This motivates the following definition.

\begin{definition}
Invariant differentiation of the jet coordinates, $J^{\,i}$ and $I^\alpha_\mathrm{K}$, are defined, respectively, as 
\begin{eqnarray}\label{diffinvariant}
\mathcal{D}_j J^{\,i}=\delta_{ij}+N_{ij},\qquad \mathcal{D}_jI^\alpha_\mathrm{K}=I^\alpha_{\mathrm{K}j}+M^\alpha_{\mathrm{K}j},
\end{eqnarray}
where $\delta_{ij}$ is the Kronecker delta, and $N_{ij}$ and $M^\alpha_{\mathrm{K}j}$ are the \textnormal{correction terms}. 
\end{definition}

Theorem \ref{errorterms} provides formulae for the correction terms $N_{ij}$ and $M^\alpha_{\mathrm{K}j}$, for which we need to define the following notion of the infinitesimal of a prolonged group action.

Let $G$ be a group parametrized by $a_1,...,a_r$, where $r=\mathrm{dim}(G)$, in a neighbourhood of the identity element. The \textit{infinitesimals of the prolonged group action} with respect to these parameters are
\begin{equation}\label{infinitesimals}
\xi^i_j=\left.\frac{\partial \widetilde{x_i}}{\partial a_j} \right|_{g=e},\qquad \phi^\alpha_{\mathrm{K},j}=\left.\frac{\partial \widetilde{u^\alpha_\mathrm{K}}}{\partial a_j}\right|_{g=e}.
\end{equation}

Since $\xi^i_j$ and $\phi^\alpha_{\mathrm{K},j}$ are functions of the $x_i$, for $i=1,...,p$, $u^\alpha$, for $\alpha=1,...,q$, and $u^\alpha_\mathrm{K}$, we can define 
$$\xi^i_j(I)=\xi^i_j(J^{\,i},I^\beta)$$
and
$$\phi^\alpha_{K,j}(I)=\phi^\alpha_{K,j}(J^{\,i},I^\beta,I^\beta_\mathrm{L}),$$
where the arguments have been invariantized.

\begin{theorem}\label{errorterms}
For a left action on the base space and a right moving frame, the $p\times r$ \textnormal{correction matrix} $\mathsf{K}$, which provides the correction terms, is given by
$$\mathsf{K}_{j\ell}=\left.\widetilde{D_j}\rho_\ell(\widetilde{z})\right|_{g=\rho(z)}=\left((T_eR_{\rho})^{-1})\mathcal{D}_j\rho\right)_\ell,$$
where the frame $\rho=(\rho_1,...,\rho_r)^T$ is in parameter form and $R_\rho:G\rightarrow G$ is right multiplication by $\rho$.
The formulae for the correction terms are
$$N_{ij}=\sum_{\ell=1}^r\mathsf{K}_{j\ell}\xi^i_\ell(I),\qquad M^\alpha_{\mathrm{K}j}=\sum_{\ell=1}^r\mathsf{K}_{j\ell}\phi^\alpha_{\mathrm{K},\ell}(I),$$
where $\ell$ is the index for the group parameters and $r=\dim(G)$. 
\end{theorem}

The proof of this theorem can be found in page 134 of \cite{Mansfield}. 

The error terms can be calculated without explicit knowledge of the frame, requiring merely information on the normalization equations and infinitesimals -- symbolic software exists which computes these, see \cite{AIDA} among others. From Equation (\ref{diffinvariant}), one can verify that the processes of invariantization and differentiation do not commute. If we consider two generating invariants, $I^\alpha_\mathrm{J}$ and $I^\alpha_\mathrm{L}$, and let $\mathrm{JK}=\mathrm{LM}$ such that $I^\alpha_{\mathrm{JK}}=I^\alpha_{\mathrm{LM}}$, then we obtain the so-called \textit{syzygies} or \textit{differential identities}
\begin{equation}
\mathcal{D}_\mathrm{K} I^\alpha_\mathrm{J}-\mathcal{D}_\mathrm{M} I^\alpha_\mathrm{L}=M^\alpha_{\mathrm{JK}}-M^\alpha_{\mathrm{LM}}.
\end{equation}
For more information on syzygies, see Chapter 5 in \cite{Mansfield}. A full discussion of the finite generation of invariant differential algebras and their syzygy modules is given in 
 \cite{Hubert,HubertKogan}.

\vspace{0.4cm} 
\noindent\textbf{Example \ref{linearsl2} (cont.)}
\textit{The invariant differential operators for this action are}
\begin{eqnarray}\label{MainEGInvDiffOps}
\mathcal{D}_x &\kern-8pt=&\kern-8pt x\displaystyle\frac{\mathrm{d}}{\mathrm{d} x} + y\displaystyle\frac{\mathrm{d}}{\mathrm{d} y},\\\label{MainEGInvDiffOps1}
\mathcal{D}_y &\kern-8pt=&\kern-8pt -\displaystyle\frac{u_y}{xu_x+yu_y}\frac{\mathrm{d}}{\mathrm{d} x} + 
\displaystyle\frac{u_x}{xu_x+yu_y}\frac{\mathrm{d}}{\mathrm{d} y}.
\end{eqnarray}
\textit{It can now be seen that in the list of differential invariants given in Equation (\ref{basicExListofDiffInvs}), that the fourth comnponent is $\mathcal{D}_x(u)$, the sixth component is $\mathcal{D}_x^2(u)-\mathcal{D}_x(u)$, and the seventh component is $\mathcal{D}_y\mathcal{D}_x(u)$. It is not possible, however, to obtain the eighth component, $I(u_{yy})$ by invariant differentiation of $u$, since $\mathcal{D}_y(u)=0$. All other differential invariants of the form $I(u_K)$ can be obtained from $u$ and $I(u_{yy})$ by invariant differentiation and algebraic operations, and thus these two invariants  generate the algebra of invariants.}

\textit{The syzygy between $I(u)$ and $I(u_{yy})$ is} 
\begin{equation}\label{MainExSyzA}
\mathcal{D}_x(I(u_{yy})) - \mathcal{D}_y^2\mathcal{D}_x(u)=-4I(u_{yy})+\frac1{\mathcal{D}_x(u)}
\left( I(u_{yy})\mathcal{D}_x^2(u) -2 \left(\mathcal{D}_y\mathcal{D}_x(u)\right)^2   \right).
\end{equation}

 \begin{example}\label{examplewithplot} 
{We now extend the previous example by adding an extra, dummy, independent variable $\tau$, which we declare
to be invariant under the group action. In the sequel, we will  use 
differentiation by $\tau$ to effect the variation, a step which will allow us to use the invariant calculus to achieve our results.  As $\tau$ is a dummy variable, the normalisation equations will never contain $\tau$ derivatives. The new generating invariants will therefore be first order in $\tau$, and there will be new syzygies.}
Set $u=u(x,y,\tau)$. Let $g\in SL(2)$ act on $(x,y,u(x,y,\tau))$ as in Example \ref{linearsl2} and set $\widetilde{\tau}=\tau$. Taking the normalization equations as before, we obtain
\begin{equation}\nonumber\begin{array}{rcl}
\widetilde{u_\tau}|_{g=\rho(\zede)}\kern-8pt&=I^u_3\kern -8pt&=u_\tau,\\[8pt]
\widetilde{u_{xx}}|_{g=\rho(\zede)}\kern -8pt&=I^u_{11}\kern-8pt&=x^2u_{xx}+2xyu_{xy}+y^2u_{yy},\\[8pt]
\widetilde{u_{xy}}|_{g=\rho(\zede)}\kern -8pt &=I^u_{12}\kern -8pt&=\displaystyle{\frac{xu_xu_{xy}-yu_yu_{xy}+yu_xu_{yy}-xu_yu_{xx}}{xu_x+yu_y}},\\[8pt]
\widetilde{u_{yy}}|_{g=\rho(\zede)}\kern -8pt &=I^u_{22}\kern -8pt &= \displaystyle{\frac{u_x^2u_{yy}-2u_xu_yu_{xy}+u_y^2u_{xx}}{(xu_x+yu_y)^2}}.
 \end{array}\end{equation}
From Figure \ref{fig:unicafig}, we can see that there are two ways to reach $I^u_{113}$ and since these must yield the same result, we get the following syzygy between $I^u_3$ and $I^u_{11}$:
\begin{equation}\label{sysIu11}
\mathcal{D}_\tau I^u_{11}=\mathcal{D}_x^2 I^u_3-\mathcal{D}_x I^u_3.
\end{equation}
Similarly, there are two possibilities to obtain $I^u_{223}$, which give rise to the following syzygy between $I^u_3$ and $I^u_{22}$:
\begin{equation}\label{sysIu22}
\mathcal{D}_\tau I^u_{22}=\mathcal{D}_y^2 I^u_3-\dfrac{2 I^u_{12}}{I^u_1}\mathcal{D}_y I^u_3+\dfrac{I^u_{22}}{I^u_1}\mathcal{D}_x I^u_3.
\end{equation}
Finally, there are several ways in which to reach $I^u_{123}$; there are two syzygies between $I^u_3$ and $I^u_{12}$, which are as follows:
\begin{eqnarray}\label{sysIu12a}
\mathcal{D}_\tau I^u_{12}\kern -8pt&\kern -58pt= \mathcal{D}_y\mathcal{D}_x I^u_3-\left(\dfrac{I^u_{11}}{I^u_1}+1\right)\mathcal{D}_y I^u_3,\\ \label{sysIu12b}
\mathcal{D}_\tau I^u_{12}\kern -8pt&=\mathcal{D}_x\mathcal{D}_y I^u_3+\left(1-\dfrac{I^u_{11}}{I^u_1}\right)\mathcal{D}_y I^u_3+\dfrac{I^u_{12}}{I^u_1}\mathcal{D}_x I^u_3.
\end{eqnarray}
 \end{example}
 
 \begin{figure*}[h]
\begin{center}
\begin{pspicture}(-2,-3)(4,4)
\psset{unit=1.5}
\psline[linewidth=0.2pt,linestyle=dashed,arrowscale=2]{->}(0,0)(0,2.5)
\psline[linewidth=0.2pt,linestyle=dashed,arrowscale=2]{->}(0,0)(4,0)
\psline[linewidth=0.2pt,linestyle=dashed,arrowscale=2]{->}(0,0)(-2.12,-2.12)
\pscircle*[linecolor=black](0,0){0.075}
\rput(-0.4,0.05){$I^u$}
\pscircle*[linecolor=black](1,0){0.075}
\rput(1,-0.4){$I^u_2$}
\rput{90}(1,-0.7){$=$}
\rput(1,-1){$0$}
\pscircle*[linecolor=black](2,0){0.075}
\rput(2,-0.4){$I^u_{22}$}
\pscircle*[linecolor=black](-0.7,-0.7){0.075}
\rput(-1.1,-0.65){$I^u_1$}
\pscircle*[linecolor=black](-1.41,-1.41){0.075}
\rput(-1.81,-1.36){$I^u_{11}$}
\pscircle*[linecolor=black](0,1){0.075}
\rput(-0.4,1.05){$I^u_3$}
\pscircle*[linecolor=black](0.3,-0.7){0.075}
\rput(0.3,-1.1){$I^u_{12}$}
\pscircle*[linecolor=OrangeRed](0.3,0.3){0.075}
\psline[linewidth=0.2pt,linecolor=OrangeRed,linestyle=dashed,arrowscale=2]{->}(0.075,1)(1,1)(0.35,0.35)
\psline[linewidth=0.2pt,linecolor=OrangeRed,linestyle=dashed,arrowscale=2]{->}(-0.04,0.96)(-0.7,0.3)(0.23,0.3)
\psline[linewidth=0.2pt,linecolor=OrangeRed,linestyle=dashed,arrowscale=2]{->}(0.3,-0.63)(0.3,0.23)
\rput(0.7,0.3){$I^u_{123}$}
\pscircle*[linecolor=RoyalBlue](2,1){0.075}
\psline[linewidth=0.2pt,linecolor=RoyalBlue,linestyle=dashed,arrowscale=2]{->}(2,0.07)(2,0.93)
\psline[linewidth=0.2pt,linecolor=RoyalBlue,linestyle=dashed,arrowscale=2]{->}(0.075,1.02)(1.93,1.02)
\rput(2.4,1.05){$I^u_{223}$}
\pscircle*[linecolor=ForestGreen](-1.41,-0.4){0.075}
\psline[linewidth=0.2pt,linecolor=ForestGreen,linestyle=dashed,arrowscale=2]{->}(-0.07,0.95)(-1.38,-0.35)
\psline[linewidth=0.2pt,linecolor=ForestGreen,linestyle=dashed,arrowscale=2]{->}(-1.41,-1.34)(-1.41,-0.47)
\rput(-1.81,-0.35){$I^u_{113}$}
\end{pspicture}
\textbf{\caption{Paths to the $\boldsymbol{I^u_{\mathrm{K}3}}$ in Example \ref{examplewithplot}, where $\boldsymbol{\mathrm{K}}$ represents the index of differentiation with respect to the $\boldsymbol{x_i}$, for $\boldsymbol{i=1,...,p}$.\label{fig:unicafig}}}
\end{center}
\end{figure*}

From Equations (\ref{sysIu12a}) and (\ref{sysIu12b}) in Example \ref{examplewithplot}, one can verify that the invariant operators $\mathcal{D}_x$ and $\mathcal{D}_y$ do not commute. In general, the invariant total differentiation operators do not commute. In \cite{FelsOlverII}, Fels and Olver gave a formula for the commutators of these invariant operators, which only relies on the correction matrix $\mathsf{K}$ and the infinitesimals of the group action. Denote the invariantized derivatives of the infinitesimals $\xi^k_\ell$, for $k=1,...,p$ and $\ell=1,...,r$, by
$$\Xi_{\ell i}^k=\widetilde{D_i}\xi^k_\ell(\widetilde{\zede})|_{g=\rho(\zede)}.$$
Then the commutators are given by
\begin{equation}\label{commutators}
[\mathcal{D}_i,\mathcal{D}_j]=\sum_k\mathcal{A}_{ij}^k\mathcal{D}_k, \qquad \mathcal{A}_{ij}^k=\sum_{\ell=1}^r\mathsf{K}_{j\ell}\Xi_{\ell i}^k-\mathsf{K}_{i\ell}\Xi_{\ell j}^k.
\end{equation}

Invariant Lagrangians are  invariant volume forms, which are obtained by taking the wedge product of  invariant zero and one-forms.  We define the latter next, and their behaviour under the invariant Lie derivative operators. 

\begin{definition}
The \textnormal{invariant one-forms} obtained via the moving frame are denoted as 
\begin{equation}\label{defofinvariantform}
I(\mathrm{d}x_i)=\mathrm{d}\widetilde{x_i}|_{g=\rho(\zede)}=\left.\left(\sum_{j=1}^p D_j (\widetilde{x_i})\mathrm{d}x_j\right)\right|_{g=\rho(\zede)}.
\end{equation}
\end{definition}

As for differential invariants, the invariant total differentiation operators send invariant differential forms to invariant differential forms. 

Let the invariant differential operator $\mathcal{D}_i$ be associated to the vector field $\mathbf{V}_i$ as follows
\begin{equation}\label{association}
\mathcal{D}_i=f_1(\zede)D_1+\cdots f_p(\zede)D_p \longleftrightarrow \mathbf{V}_i=( f_1(\zede),...,f_p(\zede)).
\end{equation}

Consider the invariant total differentiation $\mathcal{D}_i$ of a form $\omega$, denoted as $\mathcal{D}_i(\omega)$, to be the \textit{Lie derivative}
\begin{equation}\label{bypropofliederiv}
\mathcal{D}_i(\omega)=\mathrm{d} (\mathbf{V}_i\lrcorner\,\omega)+\mathbf{V}_i\lrcorner\,(\mathrm{d}\omega),
\end{equation}
where  $\mathrm{d}$ is the usual exterior derivative, and $\lrcorner$ is the interior product of a vector field with a form. In fact if $\omega=I(\mathrm{d}x_j)$, then (\ref{bypropofliederiv}) simplifies to
\begin{equation}\label{liederivdual}
\mathcal{D}_i\left(I(\mathrm{d}x_j)\right)=\mathbf{V}_i\lrcorner\, \left(\mathrm{d}\, I(\mathrm{d}x_j)\right),
\end{equation}
by the following lemma.

\begin{lemma}\label{lemmadualbasis}
Let $\mathbf{V}_i$ be the vector associated to the invariant total differentiation operator $\mathcal{D}_i$. Then 
\begin{equation}\label{dualbasis}
\mathbf{V}_i\lrcorner\,I(\mathrm{d}x_j)=\delta_{ij},
\end{equation}
where $\delta_{ij}$ is the Kronecker delta, in other words $\{I(\mathrm{d}x_1),...,I(\mathrm{d}x_p)\}$ forms a basis to the dual space of $TM|_{\widetilde{\mathrm{x}}}$, whose basis is  $\{\mathcal{D}_1,...,\mathcal{D}_p\}$.
\end{lemma}

\begin{proof}
Let $\mathcal{J}$ denote the Jacobian matrix $\mathrm{d}\widetilde{\mathbf{x}}/\mathrm{d}\mathbf{x}|_{g=\rho(\zede)}$. Then
$$\begin{array}{rl}
\mathbf{V}_i\lrcorner\, I(\mathrm{d}x_j)\kern -8pt &= \displaystyle{\left((\mathcal{J}^{-T})_{i1},...,(\mathcal{J}^{-T})_{ip}\right)\lrcorner\, \Big(\sum_{\ell=1}^p (\mathcal{J})_{j\ell} \mathrm{d}x_\ell\Big)}\\
\kern -8pt &= \displaystyle{\left((\mathcal{J}^{-1})_{1i},...,(\mathcal{J}^{-1})_{pi}\right)\lrcorner\, \Big(\sum_{\ell=1}^p (\mathcal{J})_{j\ell} \mathrm{d}x_\ell\Big)}\\ 
\kern -8pt &= (\mathcal{J}^{-1})_{1i} (\mathcal{J})_{j1}+\cdots + (\mathcal{J}^{-1})_{pi} (\mathcal{J})_{jp}\\
\kern -8pt &=\delta_{ij}.
\end{array}$$
$\hfill \Box$
\end{proof}

It is possible to calculate the Lie derivative of the $I(\mathrm{d}x_j)$ with respect to the $\mathcal{D}_i$ knowing only  the infinitesimals and the normalization equations, that is, without explicit knowledge of the frame. The following theorem shows exactly this.

\begin{theorem}\label{liederivativeofform}
Let $g\in G$ act on $\mathbf{x}\in X$ and let $f$ be a function in $M$, and denote the set of invariant total differentiation operators by $\{\mathcal{D}_i\}$, and the set of invariant one-forms, $\{I(\mathrm{d}x_j)\}$. 
Then setting
\begin{equation}\label{formulaforliederiv}
\mathcal{D}_i( I(\mathrm{d}x_j))=\sum_{k=1}^p\mathcal{B}^k_{ij} I(\mathrm{d}x_k)
\end{equation}
we have
$$\mathcal{B}_{ki}^j=\mathcal{A}_{jk}^i,$$
and
$$[\mathcal{D}_j,\mathcal{D}_k](f)=\sum_{i=1}^p\mathcal{A}_{jk}^i\mathcal{D}_i(f)$$
where the $\mathcal{A}_{jk}^i$ are given explicitly  in (\ref{commutators}).   
\end{theorem}

\begin{proof} 
We  next prove that for any function $f\in M$,
$$\mathrm{d}f=\sum_{i=1}^p\mathcal{D}_i(f) I(\mathrm{d}x_i).$$
Let $\mathrm{d}\mathbf{x}=(\mathrm{d}x_1,...,\mathrm{d}x_p)^T$ and $\boldsymbol{D}=(D_1,...,D_p)^T$; further, set  $I(\mathrm{d}\mathbf{x})=(I(\mathrm{d}x_1),...,I(\mathrm{d}x_p))^T$  and $\boldsymbol{\mathcal{D}}=(\mathcal{D}_1,...,\mathcal{D}_p)^T$. We know that $I(\mathrm{d}\mathbf{x})=\mathcal{J}\mathrm{d}\mathbf{x}$, where $\mathcal{J}$ is the Jacobian matrix $\mathrm{d}\widetilde{\mathbf{x}}/\mathrm{d}\mathbf{x}|_{g=\rho(\zede)}$, so that $\mathrm{d}\mathbf{x}=\mathcal{J}^{-1}I(\mathrm{d}\mathbf{x})$, $\boldsymbol{\mathcal{D}}=\mathcal{J}^{-T}\boldsymbol{D} $ and $\boldsymbol{D}=\mathcal{J}^T\boldsymbol{\mathcal{D}}$, then
$$\begin{array}{rl}
\mathrm{d} f\kern -8pt &= \displaystyle{\sum_{n=1}^p\frac{\partial f}{\partial x_n}\mathrm{d}x_n}\\[10pt]
\kern -8pt &= \displaystyle{\sum_{n=1}^p\left[\sum_{m=1}^p\left(\mathcal{J}^T\right)_{nm}\mathcal{D}_m(f)\left(\sum_{i=1}^p (\mathcal{J}^{-1})_{ni} I(\mathrm{d}x_i)\right)\right]}\\[10pt]
\kern -8pt &= \displaystyle{\sum_{i=1}^p\sum_{m=1}^p\sum_{n=1}^p (\mathcal{J})_{mn}(\mathcal{J}^{-1})_{ni}\mathcal{D}_m(f)I(\mathrm{d}x_i)}\\[10pt]
\kern -8pt &= \displaystyle{\sum_{i=1}^p\sum_{m=1}^p\delta_{mi}\mathcal{D}_m(f)I(\mathrm{d}x_i)}\\[10pt]
\kern -8pt &=\displaystyle{\sum_{i=1}^p\mathcal{D}_i(f)I(\mathrm{d}x_i)}.
\end{array}$$

Next, since $\mathrm{d}^2\equiv 0$, we have 
$$\displaystyle{0=\mathrm{d}^2f=\mathrm{d}\left(\sum_{i=1}^p\mathcal{D}_i(f)I(\mathrm{d}x_i)\right)=\sum_{i=1}^p\left[\mathrm{d}(\mathcal{D}_i(f))\wedge I(\mathrm{d}x_i)+\mathcal{D}_i(f)\mathrm{d} (I(\mathrm{d}x_i))\right]}.$$
Let $\mathbf{V}_k$ be the vector associated to $\mathcal{D}_k$ as defined in (\ref{association}). From $\mathbf{V}_k\lrcorner\, \mathrm{d}^2f=0$, it follows that
$$\begin{array}{rl}
0\kern -8pt&=\displaystyle{\sum_{i=1}^p\left[(\mathbf{V}_k\lrcorner\, \mathrm{d})(\mathcal{D}_i(f)) I(\mathrm{d}x_i) -\mathrm{d}(\mathcal{D}_i(f))(\mathbf{V}_k\lrcorner\, I(\mathrm{d}x_i))+\mathcal{D}_i(f)(\mathbf{V}_k\lrcorner\, \mathrm{d}) (I(\mathrm{d}x_i))\right]}\\[10pt]
\kern -8pt&=\displaystyle{\sum_{i=1}^p\left[\mathcal{D}_k(\mathcal{D}_i(f)) I(\mathrm{d}x_i)-\delta_{ki}\mathrm{d}(\mathcal{D}_i(f))+\mathcal{D}_i(f)\mathcal{D}_k( I(\mathrm{d}x_i))\right]}\\[10pt]
\kern -8pt&=\displaystyle{\sum_{i=1}^p\left[\mathcal{D}_k(\mathcal{D}_i(f)) I(\mathrm{d}x_i)+\mathcal{D}_i(f)\sum_{m=1}^p\mathcal{B}_{ki}^m I(\mathrm{d}x_m)\right]-\mathrm{d}(\mathcal{D}_k(f)),}
\end{array}$$
where we have used the properties of the interior product in the first line, the equality (\ref{liederivdual}) in the second line, and the definition of $ \mathcal{B}^k_{ij}  $, (\ref{formulaforliederiv}), in the third line.
Note this proves that $\mathcal{D}_i( I(\mathrm{d}x_j)) $ is linear in the $I(\mathrm{d}x_{\ell}) $. 

Finally, we have further  that $\mathbf{V}_j\lrcorner \, (\mathbf{V}_k\lrcorner\,\mathrm{d}^2f)=0$, and thus
$$\begin{array}{rl}
0\kern -8pt &=\displaystyle{\sum_{i=1}^p \left[\mathcal{D}_k(\mathcal{D}_i(f))\delta_{ij}+\mathcal{D}_i(f)\mathcal{B}_{ki}^m\delta_{mj}\right]-(\mathbf{V}_j\lrcorner\,\mathrm{d})\mathcal{D}_k(f)}\\[10pt]
\kern -8pt &=\displaystyle{\mathcal{D}_k(\mathcal{D}_j(f))-\mathcal{D}_j(\mathcal{D}_k(f))+\sum_{i=1}^p\mathcal{D}_i(f)\mathcal{B}_{ki}^j}\\[10pt]
\kern -8pt &\displaystyle{=[\mathcal{D}_k,\mathcal{D}_j](f)+\sum_{i=1}^p\mathcal{D}_i(f)\mathcal{B}_{ki}^j},
\end{array}$$
where we have used the properties of the interior product in the first line and the equality (\ref{liederivdual}) in the second line. Rewriting the above we obtain
$$[\mathcal{D}_j,\mathcal{D}_k](f)=\sum_{i=1}^p\mathcal{D}_i(f)\mathcal{B}_{ki}^j.$$
Since $[\mathcal{D}_j,\mathcal{D}_k](f)=\sum_{i=1}^p\mathcal{A}_{jk}^i\mathcal{D}_i(f)$, where $\mathcal{A}_{jk}^i$ is defined in Equation (\ref{commutators}), this implies that
$$\mathcal{A}_{jk}^i=\mathcal{B}_{ki}^j,$$
as required. $\hfill \Box$
\end{proof}

\begin{example}\label{exampleonliederiv}
Recall that $\tau$ is an invariant dummy independent variable introduced to effect variation, a textcolor{blue{device} that will enable us to use the invariant calculus necessary for our results.} Let $g \in SL(2)$ act on $(x,y,\tau)$ as in Example \ref{examplewithplot}. Then the Lie derivatives of $I(\mathrm{d}x_j)$ with respect to $\mathcal{D}_i$ are as shown in Table \ref{tableofliederiv}.
\end{example}

\begin{table}[h]
\begin{center}
\begin{tabular}{l|c|c|c}
Lie derivative & $I(\mathrm{d}x)$ &$I(\mathrm{d}y)$ & $I(\mathrm{d}\tau)$ \\
\hline
$\mathcal{D}_x$ & $\dfrac{I^u_{12}}{I^u_1}I(\mathrm{d}y)$ & $2 I(\mathrm{d}y)$ & $0$\\
$\mathcal{D}_y$ & $-\dfrac{I^u_{12}}{I^u_1}I(\mathrm{d}x)-\dfrac{I^u_{23}}{I^u_1}I(\mathrm{d}\tau)$ & $-2 I(\mathrm{d}x)$ & $0$\\
$\mathcal{D}_\tau$ & $\dfrac{I^u_{23}}{I^u_1}I(\mathrm{d}y)$ & $0$ & $0$
\end{tabular}
\textbf{\caption{Lie derivatives of the $\boldsymbol{I(\mathrm{d}x_j)}$ with respect to the $\boldsymbol{\mathcal{D}_i}$.}\label{tableofliederiv}}
\end{center}
\end{table}

Note that in Example \ref{exampleonliederiv}, the Lie derivatives $\mathcal{D}_i$ of $I(\mathrm{d}\tau)$ are all equal to zero. This is no coincidence as is shown in the following lemma.

\begin{lemma}\label{lasttermdrops}
Let $g\in G$ act on the set of independent variables $\{x_i\}$, for $i=1,...,p+1$. If $g\cdot x_{p+1}=x_{p+1}$, then 
$$\mathcal{D}_i\left(I(\mathrm{d}x_{p+1})\right)=0,$$
for all $i=1,...,p+1$.
\end{lemma}

\begin{proof}
The Lie derivative of a form can be written as
$$\mathcal{D}_i\left(I(\mathrm{d}x_{p+1})\right)=\sum_{\ell=1}^{p+1} \mathcal{B}^\ell_{i,p+1}I(\mathrm{d}x_\ell).$$
According to Theorem \ref{liederivativeofform}, the coefficients $ \mathcal{B}^\ell_{i,p+1}$ are equal to 
$$\mathcal{A}_{\ell i}^{p+1}=\sum_{n=1}^r \mathsf{K}_{in}\Xi_{n\ell}^{p+1}-\mathsf{K}_{\ell n}\Xi_{ni}^{p+1}.$$
Since $x_{p+1}$ is invariant, $\xi_n^{p+1}=0$, and therefore, $ \Xi^{p+1}_{n\ell}= \Xi^{p+1}_{ni}=0$. Thus, for $\ell=1,...,p+1$, 
$$\mathcal{B}^\ell_{i,p+1}I(\mathrm{d}x_\ell)=0.$$
$\hfill \Box$
\end{proof}

As we are interested in calculating the invariantized Euler-Lagrange equations and its associated conservation laws for variational problems whose independent variables are not invariant, it will at times be necessary to apply recursively the commutators $[\mathcal{D}_{p+1},\mathcal{D}_i]=\sum_{k=1}^{p+1}\mathcal{A}^k_{p+1,i}\mathcal{D}_k$, for $i=1,...,p$, where $x_{p+1}$ is a dummy invariant independent variable and $\mathcal{A}^k_{p+1,i}$ are as defined in (\ref{commutators}) . Lemma \ref{lemmaofmcommutators} provides a formula for the commutators $[\mathcal{D}_{p+1},\mathcal{D}_\mathrm{K}]$, where $\mathrm{K}$ is a multi-index of differentiation with respect to $x_i$, for $i=1,...,p$.

\begin{lemma}\label{lemmaofmcommutators}
Let $g\in G$ act on the set of independent variables $\{x_i\}$, for $i=1,...,p+1$. If $g\cdot x_{p+1}=x_{p+1}$ and $\omega$ is some differential form on $M$, then
\begin{equation}\label{Kcommutator}
\mathcal{D}_{p+1}\mathcal{D}_\mathrm{K}(\omega)=\left(\mathcal{D}_\mathrm{K}\mathcal{D}_{p+1}+\sum_{\ell=1}^{m}\sum_{n=1}^p\mathcal{D}_{\mathrm{K}_\ell}\left(\mathcal{A}^n_{p+1,\,k_\ell}\mathcal{D}_n\right)\mathcal{D}_{\mathrm{K}\backslash(\mathrm{K}_\ell,k_\ell)}\right)(\omega),
\end{equation}
where $\mathrm{K}=(k_1,...,k_m)$ is a multi-index of differentiation with respect to $x_i$, for $i=1,...,p$, of order $m$ and, $\mathrm{K}_\ell$ and  $\mathrm{K}\backslash (\mathrm{K}_\ell,k_\ell)$ are tuples of differentiation of the following form
$$\mathrm{K}_\ell=(k_1,...,k_{\ell-1}), \quad with\quad \mathrm{K}_1=(0),\quad 
and\quad  \mathrm{K}\backslash (\mathrm{K}_\ell,k_\ell)=(k_{\ell+1},...,k_m).$$
\end{lemma}

\begin{proof}
To obtain (\ref{Kcommutator}), we use the equation for the commutators (\ref{commutators}) recursively as follows,
\begin{align}\nonumber
&\mathcal{D}_{p+1}\mathcal{D}_\mathrm{K}(\omega)\\\nonumber
&\quad=\displaystyle{\left(\mathcal{D}_{k_1}\mathcal{D}_{p+1}+\sum_{n=1}^{p+1}\mathcal{A}^n_{p+1,\,k_1}\mathcal{D}_n\right)\mathcal{D}_{k_2}\cdots\mathcal{D}_{k_m}(\omega)}\\\nonumber
&\quad=\displaystyle{\mathcal{D}_{k_1}\left(\mathcal{D}_{k_2}\mathcal{D}_{p+1}+\sum_{n=1}^{p+1}\mathcal{A}^n_{p+1,\,k_2}\mathcal{D}_n\right)\mathcal{D}_{k_3}\cdots\mathcal{D}_{k_m}(\omega)+\sum_{n=1}^{p+1}\mathcal{A}^n_{p+1,\,k_1}\mathcal{D}_n\mathcal{D}_{k_2}\cdots\mathcal{D}_{k_m}(\omega)}\\\label{laststep}
&\quad=\displaystyle{\mathcal{D}_{k_1}\mathcal{D}_{k_2}\mathcal{D}_{p+1}\mathcal{D}_{k_3}\cdots\mathcal{D}_{k_m}(\omega)+\sum_{\ell=1}^2\sum_{n=1}^{p+1}\mathcal{D}_{\mathrm{K}_\ell}(\mathcal{A}^n_{p+1,\,k_\ell}\mathcal{D}_n)\mathcal{D}_{\mathrm{K}\backslash (\mathrm{K}_\ell,k_\ell)}(\omega),}
\end{align}
and so on. Note that as $\widetilde{x_{p+1}}=x_{p+1}$, then $\xi_j^{p+1}=0$, for all $j=1,...,r$, and therefore, from (\ref{commutators}) we have that $\mathcal{A}^{p+1}_{p+1,\,k_\ell}=0$ for all $\ell$. 
After applying the commutators (\ref{commutators}) recursively and setting $\mathcal{A}^{p+1}_{p+1,\,k_\ell}$ to zero for all $\ell$, (\ref{laststep}) becomes
$$\displaystyle{\mathcal{D}_\mathrm{K}\mathcal{D}_{p+1}(\omega)=\mathcal{D}_\mathrm{K}\mathcal{D}_{p+1}(\omega)+\sum_{\ell=1}^m \sum_{n=1}^{p}\mathcal{D}_{\mathrm{K}_\ell}(\mathcal{A}^n_{p+1,\,k_\ell}\mathcal{D}_n)\mathcal{D}_{\mathrm{K}\backslash(\mathrm{K}_\ell,k_\ell)}(\omega)}.$$
$\hfill \Box$
\end{proof}

\subsection{Invariant Calculus of Variations}\label{subsecINVARIANTCALCULUS}
Consider Lagrangians to be smooth functions of $\mathbf{x}$, $\mathbf{u}$ and finitely many derivatives of $u^\alpha$ and denote them as $\bar{\mathscr{L}}[\mathbf{u}]=\int \bar{L}[\mathbf{u}]\,\mathrm{d}^p\mathbf{x}$, where $\mathrm{d}^p\mathbf{x}=\mathrm{d}x_1\dots\mathrm{d}x_p$. Moreover, assume these to be invariant under some group action and let the $\kappa_j$, for $j=1,...,N$, denote the generating differential invariants of that group action; in \cite{HubertKogan} Hubert and Kogan prove that there exists a finite number of generating invariants.  We can then rewrite $\bar{\mathscr{L}}[\mathbf{u}]$ as $\mathscr{L}[\boldsymbol{\kappa}]=\int L[\boldsymbol{\kappa}]\, I(\mathrm{d}^p\mathbf{x})$, where $I(\mathrm{d}^p\mathrm{x})=I(\mathrm{d}x_1)\dots I(\mathrm{d}x_p)$ is the invariant volume form obtained via the moving frame.

Kogan and Olver in \cite{KoganOlver} obtained formulae for the invariantized Euler-Lagrange equations through the construction of a variational bicomplex; we arrive at these using calculations that are similar to those employed to obtain the Euler-Lagrange equations in the original variables $(\mathbf{x},\mathbf{u})$. 

Recall that if $\mathbf{x}\mapsto (\mathbf{x},\mathbf{u}(\mathbf{x}))$ extremizes the functional $\bar{\mathscr{L}}[\mathbf{u}]$, then a small perturbation of $\mathbf{u}$ yields
\begin{eqnarray}
 0&\kern -155pt =\displaystyle{\left.\frac{\mathrm{d}}{\mathrm{d}\varepsilon}\right|_{\varepsilon=0}\bar{\mathscr{L}}[\mathbf{u}+\varepsilon\mathbf{v}]}\nonumber \\[10pt]
&=\displaystyle{\int \sum_{\alpha=1}^q\left[\mathsf{E}^\alpha(\bar{L})v^\alpha+\sum_{i=1}^p\frac{\mathrm{d}}{\mathrm{d}x_i}\left(\frac{\partial \bar{L}}{\partial u^\alpha_i}v^\alpha+\cdots\right)\right]\mathrm{d}^p\mathbf{x}}\nonumber
\end{eqnarray}
after differentiation under the integral sign and integration by parts, where
\begin{eqnarray}
\mathsf{E}^\alpha=\sum_\mathrm{K}(-1)^{m}\frac{\mathrm{d}^{m}}{\mathrm{d}x_{k_1}\dots \mathrm{d}x_{k_m}}\frac{\partial}{\partial u^\alpha_\mathrm{K}}\nonumber
\end{eqnarray}
is the Euler operator with respect to the dependent variables $u^\alpha$ and $\mathrm{K}=(k_1,...,k_m)$.

To obtain the invariantized analogue of $\left.\frac{\mathrm{d}}{\mathrm{d}\varepsilon}\right|_{\varepsilon=0}\bar{\mathscr{L}}[\mathbf{u}+\varepsilon\mathbf{v}]$, where the Lagrangian is given in terms of the differential invariants and an invariant volume form, we must first introduce a dummy invariant independent variable $x_{p+1}$, where $p$ is the number of independent variables.

The introduction of this new independent variable results in $q$ new invariants $I^\alpha_{p+1}=g\cdot \partial u^\alpha/\partial x_{p+1} |_{g=\rho(\zede)}$ and a set of syzygies $\mathcal{D}_{p+1} \boldsymbol{\kappa}=\mathcal{H} I(\mathbf{u_{p+1}})$, that is
\begin{equation}\label{syzygies}
\mathcal{D}_{p+1} \left(\begin{array}{c}\kappa_1\\ \vdots\\ \kappa_N\end{array}\right)=\mathcal{H}\left(\begin{array}{c} I^1_{p+1}\\ \vdots \\ I^q_{p+1}\end{array}\right),
\end{equation}
where $\mathcal{H}$ is an $N\times q$ matrix of operators depending only on the $\mathcal{D}_i$, for $i=1,...,p$, the $\kappa_j$, for $j=1,...,N$, and their invariant derivatives. Since the independent variables are not necessarily invariant, the operators $\mathcal{D}_i$, for $i=1,...,p$, and  $\mathcal{D}_{p+1}$ do not commute in general.

We know that, symbolically,
$$\left.\frac{\mathrm{d}}{\mathrm{d}\varepsilon}\right|_{\varepsilon=0}\bar{\mathscr{L}}[\mathbf{u}+\varepsilon\mathbf{v}]\quad=\left.\frac{\mathrm{d}}{\mathrm{d} \tau}\right|_{\mathbf{u}_\tau=\mathbf{v}}\bar{\mathscr{L}}[\mathbf{u}].$$

Proceeding as for the calculation of the Euler-Lagrange equations in the original variables, we obtain the following, after differentiating under the integral sign and performing integration by parts,
\begin{align}\nonumber
0=& \,\displaystyle{\mathcal{D}_{p+1} \int L[\boldsymbol{\kappa}] I(\mathrm{d}^p\mathrm{x})}\\ \nonumber
= & \displaystyle{\int \Big[\sum_{j,\mathrm{K}}\frac{\partial L}{\partial \mathcal{D}_\mathrm{K}\kappa_j}\mathcal{D}_{p+1}\mathcal{D}_\mathrm{K}\kappa_j I(\mathrm{d}^px)+L\mathcal{D}_{p+1}(I(\mathrm{d}^px))\Big]}\\ \nonumber
=& \displaystyle{\int \Big[\sum_{j,\mathrm{K}}\frac{\partial L}{\partial \mathcal{D}_\mathrm{K}\kappa_j}\Big(\mathcal{D}_\mathrm{K}\mathcal{D}_{p+1}+\sum_{\ell=1}^{m}\sum_{i=1}^p\mathcal{D}_{\mathrm{K}_\ell}(\mathcal{A}^i_{p+1,k_\ell}\mathcal{D}_i)\mathcal{D}_{\mathrm{K}\backslash(\mathrm{K}_\ell,k_\ell)}\Big)\left(\kappa_j I(\mathrm{d}^p\mathrm{x})\right)+L\mathcal{D}_{p+1}(I(\mathrm{d}^p\mathrm{x}))\Big]}\\ \nonumber
=& \displaystyle{\int \Big[\sum_{j,\mathrm{K}}\Big((-1)^m\mathcal{D}_\mathrm{K}\Big(\frac{\partial L}{\partial \mathcal{D}_\mathrm{K}\kappa_j} I(\mathrm{d}^p\mathrm{x})\Big)\uwave{\mathcal{D}_{p+1}\kappa_j}}\\\nonumber
&\quad\displaystyle{+\frac{\partial L}{\partial \mathcal{D}_\mathrm{K}\kappa_j}\sum_{\ell=1}^{m}\sum_{i=1}^p\mathcal{D}_{\mathrm{K}_\ell}(\mathcal{A}^i_{p+1,k_\ell}\mathcal{D}_i)\mathcal{D}_{\mathrm{K}\backslash (\mathrm{K}_\ell,k_\ell)}\left(\kappa_j I(\mathrm{d}^p\mathrm{x})\right)\Big)}\\\label{oneofmany}
&\quad\displaystyle{+L\sum_{j=1}^p I(\mathrm{d}x_1)...\mathcal{D}_{p+1} I(\mathrm{d}x_j)...I(\mathrm{d}x_p)\Big]+\mbox{B.T.'s,}}
\end{align}
where B.T.'s stands for boundary terms, $m$ is the order of the multi-index of differentiation $\mathrm{K}$, and $\mathrm{K}_\ell$ and $\mathrm{K}\backslash(\mathrm{K}_\ell,k_\ell)$ correspond to the tuples defined in Lemma \ref{lemmaofmcommutators}. Note that we have used Lemma \ref{lemmaofmcommutators} in (\ref{oneofmany}).

Next, we substitute the underlined $\mathcal{D}_{p+1}\kappa_j$ by (\ref{syzygies}) and use Theorem \ref{liederivativeofform} to differentiate the invariant one-forms, which yields
\begin{align}\nonumber
& \displaystyle{0=\int \Big[\sum_{j,\,\mathrm{K}}\Big(\sum_\alpha\Big((-1)^m\mathcal{D}_\mathrm{K}\Big(\frac{\partial L}{\partial \mathcal{D}_\mathrm{K}\kappa_j} I(\mathrm{d}^p\mathrm{x})\Big)\mathcal{H}_{j,\,\alpha}I^\alpha_{p+1} \Big)}\\\label{firstBTs}
&\qquad\displaystyle{+\frac{\partial L}{\partial\mathcal{D}_\mathrm{K}\kappa_j}\sum_{\ell=1}^{m}\sum_{i=1}^p\mathcal{D}_{\mathrm{K}_\ell}(\mathcal{A}^i_{p+1,k_\ell}\mathcal{D}_i)\mathcal{D}_{\mathrm{K}\backslash (\mathrm{K}_\ell,k_\ell)}\kappa_j I(\mathrm{d}^p\mathrm{x})\Big)+L\sum_{j=1}^p\mathcal{B}_{p+1,j}^j I(\mathrm{d}^p\mathrm{x})\Big]+\mbox{B.T.'s.}}
\end{align}
Note that the terms $\mathcal{A}_{p+1,\,k_\ell}^i$, $\mathcal{D}_{\mathrm{K}_\ell}(\mathcal{A}_{p+1,k_\ell}^i)$, and $\mathcal{B}^j_{p+1,j}$ involve sums of terms which include $I^\alpha_{\mathrm{K},p+1}$. Unless $|\mathrm{K}|=0$, then one needs to substitute the $I^\alpha_{\mathrm{K},p+1}$, by their respective differential formulae $\mathcal{D}_\mathrm{K} I^\alpha_{p+1}-M^\alpha_{p+1,\mathrm{K}}$ -- where $M^\alpha_{p+1,\mathrm{K}}$ are the error terms obtained by applying $\mathcal{D}_\mathrm{K}$ to $I^\alpha_{p+1}$. Note that if the $M^\alpha_{p+1,\mathrm{K}}$ involve terms of the form $I^\alpha_{\mathrm{J},p+1}$, then these must also be substituted by their respective differential formulae. Performing a second set of integration by parts to (\ref{firstBTs}) yields
 \begin{equation}\label{secondBTs}
0=\int \Bigg(\sum_\alpha\mathsf{E}^\alpha(L) I^\alpha_{p+1} I(\mathrm{d}^p\mathrm{x})+\sum_{i=1}^p \mathcal{D}_i \Big(\sum_{j=1}^{p+1} F_{ij}\; I(\mathrm{d}x_1)...\widehat{I(\mathrm{d}x_j)}...I(\mathrm{d}x_{p+1})\Big)\Bigg),
\end{equation}
where $\mathsf{E}^\alpha(L)$ are the invariantized Euler-Lagrange equations as defined in (\ref{IELKO}), $F_{ij}$ depend on $I^\alpha_{\mathrm{K},p+1}$ and $I^\alpha_\mathrm{J}$ with $\mathrm{K}$ and $\mathrm{J}$ multi-indices of differentiation with respect to $x_i$, for $i=1,...,p$, and
$$I(\mathrm{d}x_1)...\widehat{I(\mathrm{d}x_j)}...I(\mathrm{d}x_{p+1})= I(\mathrm{d}x_1)...I(\mathrm{d}x_{j-1}) I(\mathrm{d}x_{j+1})...I(\mathrm{d}x_{p+1}).$$

Note that after the second set of integration by parts has been performed in (\ref{firstBTs}), all $p$-forms involving $I(\mathrm{d}x_{p+1})$, which sit outside the boundary terms, have been discarded as there is no integration along $x_{p+1}$. In the next theorem, we will show that the boundary terms of (\ref{secondBTs}) do not contain any $(p-1)$-forms involving $I(\mathrm{d}x_{p+1})$, and therefore as they crop up in the calculation we can simply just discard them. Furthermore, an important point of the next theorem is to show that the resulting boundary terms are linear in $I^\alpha_{\mathrm{K},p+1}$.

\begin{theorem}
The process of calculating the invariantized Euler-Lagrange equations produces boundary terms that can be written as 
\begin{align}
 \displaystyle{\int\sum_{i=1}^p \mathrm{d}\left((-1)^{i-1}\left(\sum_{\mathrm{K},\alpha} I^\alpha_{\mathrm{K},p+1}C^\alpha_{\mathrm{K},i}\right)I(\mathrm{d}x_1)...\widehat{I(\mathrm{d}x_i)}...I(\mathrm{d}x_p)\right),}\label{boundaryterms}
\end{align}
where 
$$ I(\mathrm{d}x_1)...\widehat{I(\mathrm{d}x_i)}...I(\mathrm{d}x_p)= I(\mathrm{d}x_1)... I(\mathrm{d}x_{i-1}) I(\mathrm{d}x_{i+1})... I(\mathrm{d}x_p),$$
$\mathrm{K}$ is a multi-index of differentiation with respect to $x_i$, for $i=1,...,p$, and $C^\alpha_{\mathrm{K},i}$ are functions of $I^\alpha_\mathrm{J}$, with $\mathrm{J}$ a multi-index of differentiation with respect to $x_i$.
\end{theorem}

\begin{proof}
Consider the boundary terms in (\ref{secondBTs})
\begin{equation}\label{BT1}
\int \sum_{i=1}^p \mathcal{D}_i \left(\sum_{j=1}^{p+1} F_{ij}\; I(\mathrm{d}x_1)...\widehat{I(\mathrm{d}x_j)}...I(\mathrm{d}x_{p+1})\right).
\end{equation}
Since $\mathcal{D}_i$ is a derivation, we obtain
\begin{align}\nonumber
&\displaystyle{\mathcal{D}_i \left(\sum_{j=1}^{p+1} F_{ij}\; I(\mathrm{d}x_1)...\widehat{I(\mathrm{d}x_j)}...I(\mathrm{d}x_{p+1})\right)}\\ \label{oneBTterm}
&\quad\displaystyle{=\sum_{j=1}^{p+1}\left(\mathcal{D}_i (F_{ij})I(\mathrm{d}x_1)...\widehat{I(\mathrm{d}x_j)}...I(\mathrm{d}x_{p+1})+F_{ij}\mathcal{D}_i(I(\mathrm{d}x_1)...\widehat{I(\mathrm{d}x_j)}...I(\mathrm{d}x_{p+1}))\right).}
\end{align}
For $j=1,...,p+1$, $\mathcal{D}_i(I(\mathrm{d}x_1)...\widehat{I(\mathrm{d}x_j)}...I(\mathrm{d}x_{p+1}))$ in (\ref{oneBTterm}) can be written as
\begin{equation}\label{liederivofpform}
\mathcal{D}_i(I(\mathrm{d}x_1))...\widehat{I(\mathrm{d}x_j)}...I(\mathrm{d}x_{p+1})+\cdots+I(\mathrm{d}x_1)...\widehat{I(\mathrm{d}x_j)}...\mathcal{D}_i(I(\mathrm{d}x_{p+1})).
\end{equation}
For $j=1,...,p$, the last term in (\ref{liederivofpform}) is zero by Lemma \ref{lasttermdrops}, also all remaining terms in (\ref{liederivofpform}) disappear as they all possess a $I(\mathrm{d}x_{p+1})$ form and there is no integration along $x_{p+1}$.  

Furthermore, for $j=1,...,p$, the terms $\mathcal{D}_i (F_{ij})I(\mathrm{d}x_1)...\widehat{I(\mathrm{d}x_j)}...I(\mathrm{d}x_{p+1})$ in (\ref{oneBTterm}) disappear as there is no integration along $x_{p+1}$. Hence, (\ref{oneBTterm}) reduces to 
\begin{align}\nonumber
&\mathcal{D}_i (F_{i,p+1})I(\mathrm{d}^p\mathrm{x})+F_{i,p+1}\mathcal{D}_i(I(\mathrm{d}^p\mathrm{x}))\\\nonumber
&\quad=\mathcal{D}_i(F_{i,p+1}I(\mathrm{d}^p\mathrm{x}))\\
&\quad=\mathrm{d}(\mathbf{V}_i\lrcorner\, F_{i,p+1} I(\mathrm{d}^p\mathrm{x}))+\mathbf{V}_i\lrcorner\,\mathrm{d}(F_{i,p+1} I(\mathrm{d}^p\mathrm{x})).\label{simplofvolform}
\end{align}
The invariant volume form, $I(\mathrm{d}^p\mathrm{x})$, can be written as $|\mathcal{J}|\, \mathrm{d}^p\mathrm{x}$, where as before $\mathcal{J}=\mathrm{d}\widetilde{\mathbf{x}}/\mathrm{d}\mathbf{x}|_{g=\rho(\zede)}$, and therefore (\ref{simplofvolform}) becomes
$$\displaystyle{\mathrm{d}((-1)^{i-1}F_{i,p+1} I(\mathrm{d}x_1)...\widehat{I(\mathrm{d}x_i)}...I(\mathrm{d}x_p))+\mathbf{V}_i\lrcorner\,\frac{\partial \left(F_{i,p+1} |\mathcal{J}|\right)}{\partial x_{p+1}} \mathrm{d}x_{p+1}\,\mathrm{d}^p\mathrm{x}}.$$
Since $\mathcal{D}_i$, which is associated to $\mathbf{V}_i$, does not involve any $D_{p+1}$, we will be left in the second summand with a form involving $\mathrm{d}x_{p+1}$ and as there is no integration along $x_{p+1}$ we obtain
\begin{equation}
\mathrm{d}((-1)^{i-1}F_{i,p+1} I(\mathrm{d}x_1)...\widehat{I(\mathrm{d}x_i)}...I(\mathrm{d}x_p)).
\end{equation}
From Theorem \ref{liederivativeofform}, we know that $\mathcal{B}^k_{ij}=\mathcal{A}^i_{jk}$, which is equal to $\sum_{\ell=1}^r \mathsf{K}_{k\ell}\Xi^i_{\ell j}-\mathsf{K}_{j\ell}\Xi^i_{\ell k}$. Since some of the terms in $F_{i,p+1}$ are products of the form $I^\alpha_{\mathrm{K},p+1}I^\beta_\mathrm{J}\mathcal{B}^k_{ij}$, where $k\ne p+1$, and the $\mathcal{B}^k_{ij}$  in these products never involve invariants of the form $I^\gamma_{\mathrm{L},p+1}$, the $F_{i,p+1}$ are linear combinations of the $I^\alpha_{\mathrm{K},p+1}$.

Thus, the boundary terms (\ref{BT1}) simplify to
\begin{align}\nonumber
&\displaystyle{\int\sum_{i=1}^p \mathrm{d}((-1)^{i-1}F_{i,p+1} I(\mathrm{d}x_1)...\widehat{I(\mathrm{d}x_i)}...I(\mathrm{d}x_p))}\\
&\displaystyle{\quad= \int\sum_{i=1}^p \mathrm{d}\left((-1)^{i-1}\left(\sum_{\mathrm{K},\alpha} I^\alpha_{\mathrm{K},p+1}C^\alpha_{\mathrm{K},i}\right)I(\mathrm{d}x_1)...\widehat{I(\mathrm{d}x_i)}...I(\mathrm{d}x_p)\right),}\label{naohamais}
\end{align}
where $C^\alpha_{\mathrm{K},i}$ are coefficients of the $I^\alpha_{\mathrm{K},p+1}$.
$\hfill \Box$
\end{proof}

\begin{example}\label{exampleeleqnbts} 
Consider the variational problem $\displaystyle{\iint u(u_{xx}u_{yy}-u_{xy}^2)\,\mathrm{d}x\mathrm{d}y}$, which is invariant under the action presented in Example \ref{linearsl2}. To find the invariantized Euler-Lagrange equation, introduce a dummy invariant independent variable $\tau$ and set $u=u(x,y,\tau)$. The introduction of this new independent variable results in the new invariant $\widetilde{u_\tau}|_{g=\rho(\zede)}=I^u_3$ and a set of syzygies, as computed in Example \ref{examplewithplot}. Rewriting the above variational problem in terms of the invariants of the group action yields
$$\iint I^u(I^u_{11}I^u_{22}-(I^u_{12})^2)I(\mathrm{d}x)I(\mathrm{d}y).$$

In the process of calculating the invariantized Euler-Lagrange equation and its boundary terms, we differentiate under the integral sign and obtain
\begin{equation*}\begin{array}{l}
\displaystyle{\mathcal{D}_\tau \iint  I^u(I^u_{11}I^u_{22}-(I^u_{12})^2)I(\mathrm{d}x)I(\mathrm{d}y)}\\
\displaystyle{= \iint \Big[\Big(\mathcal{D}_\tau (I^u)(I^u_{11}I^u_{22}-(I^u_{12})^2)+I^uI^u_{22}\mathcal{D}_\tau I^u_{11}+I^uI^u_{11}\mathcal{D}_\tau I^u_{22}}\\
\displaystyle{-2I^uI^u_{12}\mathcal{D}_\tau I^u_{12}\Big)I(\mathrm{d}x)I(\mathrm{d}y)+I^u(I^u_{11}I^u_{22}-(I^u_{12})^2)\mathcal{D}_\tau(I(\mathrm{d}x)I(\mathrm{d}y))\Big].}
\end{array}
\end{equation*}

Using Table \ref{tableofliederiv} we find that $\mathcal{D}_\tau(I(\mathrm{d}x)I(\mathrm{d}y))=0$. Then substituting $\mathcal{D}_\tau I^u_{11}$, $\mathcal{D}_\tau I^u_{22}$, and $\mathcal{D}_\tau I^u_{12}$ by (\ref{sysIu11}), (\ref{sysIu22}),  and (\ref{sysIu12a}), respectively, and performing integration by parts yields
$$\begin{array}{l}
\displaystyle{\iint 3\left(I^u_{11}I^u_{22}-(I^u_{12})^2\right)I^u_3 I(\mathrm{d}x)I(\mathrm{d}y)}\\
\displaystyle{\quad+\iint \Big[\mathcal{D}_x \left(\left(\left(I^uI^u_{22}-I^u_1 I^u_{22}+I^uI^u_{122}-\frac{I^uI^u_{11}I^u_{22}}{I^u_1}\right)I^u_3+I^u I^u_{22} I^u_{13}\right)I(\mathrm{d}x)I(\mathrm{d}y)\right)}\\
\displaystyle{\;\;\quad\qquad+\mathcal{D}_y \left(\left(\left(\frac{I^uI^u_{11}I^u_{12}}{I^u_1}-I^uI^u_{112}\right) I^u_3-2I^u I^u_{12} I^u_{13}+I^u I^u_{11} I^u_{23}\right)I(\mathrm{d}x)I(\mathrm{d}y)\right)\Big]},
\end{array}$$
where all forms involving $I(\mathrm{d}\tau)$ have been discarded as there is no integration along $\tau$. Thus, we obtain the invariantized Euler-Lagrange equation
$$\mathsf{E}^u(L)=3\left(I^u_{11}I^u_{22}-(I^u_{12})^2\right)=3(u_{xx}u_{yy}-u_{xy}^2),$$
as expected, and according to (\ref{naohamais}), the boundary terms can be written as
\begin{eqnarray}\nonumber
\displaystyle{\iint \mathrm{d}\Bigg(\left(\left(I^uI^u_{22}-I^u_1 I^u_{22}+I^uI^u_{122}-\frac{I^uI^u_{11}I^u_{22}}{I^u_1}\right)I^u_3+I^u I^u_{22} I^u_{13}\right)I(\mathrm{d}y)}\\ \label{vectorsC-alpha-i}
\displaystyle{-\left(\left(\frac{I^uI^u_{11}I^u_{12}}{I^u_1}-I^uI^u_{112}\right) I^u_3-2I^u I^u_{12} I^u_{13}+I^u I^u_{11} I^u_{23}\right)I(\mathrm{d}x)\Bigg),}
\end{eqnarray}
where the summands are linear in the $I^\alpha_{\mathrm{K}3}$ as expected. We will continue this example and obtain the conservation laws, see Example \ref{illustration}.

Finding the Euler-Lagrange equation in the original variables for this particular variational problem is a simple task and in this case, the invariantized version of the Euler-Lagrange equation does not simplify its calculation. However, the conservation laws contain many terms and using invariants to rewrite them, reduces them. We note that we have not used the translation invariance of this Lagrangian, and indeed we could have used the equiaffine action to study this problem. This would have led to three normalized derivative terms instead of just the one. However, we would also have had three generating differential invariants and two generating syzygies.
\end{example}

\begin{remark}
Note that in Example \ref{exampleeleqnbts} we could have substituted $\mathcal{D}_\tau I^u_{12}$ by Equation (\ref{sysIu12b}) instead of Equation (\ref{sysIu12a}), or we could even have used a combination of the two; in any case, no matter which syzygy is used the seemingly different boundary terms yield equivalent conservation laws.
\end{remark}

\section{Structure of Noether's conservation laws}\label{structureNCL}
In \cite{GoncalvesMansfield} it was shown that, for invariant Lagrangians that may be parametrized so that the independent variables are each invariant under the group action, Noether's conservation laws could be written in terms of the differential invariants of the group action and the adjoint representation of a moving frame for the Lie group action. Here we generalise this result to variational problems with independent variables that are not invariant; in this case Noether's conservation laws have a similar form as the ones presented in \cite{GoncalvesMansfield}, but with an extra factor -- the matrix representing  the group action on the space of $(p-1)$-forms, where $p$ is the number of independent variables.

\begin{example}\label{conservationlawsmonge}
Consider the $SL(2)$ action as in Example \ref{linearsl2} and the variational problem of Example \ref{exampleeleqnbts}. Applying Noether's Theorem to the variational problem and rewriting the three conservation laws in terms of the differential invariants of the group action yields
\begin{equation*}
\mathrm{d}\left(\bordermatrix{ & &\mathcal{A}d(\rho)^{-1} & \cr 
a &\displaystyle{\frac{xu_x-yu_y}{xu_x+yu_y}} & \displaystyle{-\frac{2u_xu_y}{(xu_x+yu_y)^2}} & -2xy \cr
b &\displaystyle{\frac{yu_x}{xu_x+yu_y}} & \displaystyle{\frac{u_x^2}{(xu_x+yu_y)^2}} & -y^2 \cr
c &\displaystyle{\frac{xu_y}{xu_x+yu_y}} & \displaystyle{-\frac{u_y^2}{(xu_x+yu_y)^2}} & x^2\cr}\bordermatrix{ & \boldsymbol{\upsilon}_1 & \boldsymbol{\upsilon}_2 \cr
& I^u_1I^u_{22}(I^u-I^u_1) &  I^u_1I^u_{12}(I^u-I^u_1) \cr
& -I^u I^u_1 I^u_{12} &  -I^u I^u_1 I^u_{11} \cr
& 0 & 0\cr}\right.
 \end{equation*}
 \begin{equation}\label{conslawsMA}
\left. \times \underbrace{\left(\begin{array}{cc}
 x & -y\\
 \displaystyle{\frac{u_y}{xu_x+yu_y}} &  \displaystyle{\frac{u_x}{xu_x+yu_y}}
 \end{array}\right)}_{\mathlarger{\mathsf{M}_{\mathcal{J}}}}\underbrace{\left(\begin{array}{c}
 \mathrm{d}y\\
 \mathrm{d}x
 \end{array}\right)}_{\mathlarger{\mathrm{d}^1\widehat{\mathbf{x}}}}\right)=0,
 \end{equation}
where $\mathcal{A}d(\rho)^{-1}$ is the inverse of the Adjoint representation of $SL(2)$ with respect to its generating vector fields evaluated at the frame (\ref{framesl2}), $\boldsymbol{\upsilon}_1$ and $\boldsymbol{\upsilon}_2$ are vectors of invariants, and $\mathsf{M}_{\mathcal{J}}$ is the matrix of first minors of the Jacobian matrix $\mathcal{J}$, as defined in the proof of Lemma \ref{lemmadualbasis}, evaluated at the frame  (\ref{framesl2}). \textcolor{musgo}The quantity $\mathsf{M}_\mathcal{J}\mathrm{d}^1\widehat{\mathbf{x}}$ is in fact invariant, as will be shown in the proof of Theorem \ref{conslawsnoninvindvar}, Equation (\ref{invariantp_1forms}).
\end{example}

\subsection{The  group action on the conservation laws}
Before we proceed to generalising the result in \cite{GoncalvesMansfield}, we shall look in detail at the group action on the conservation laws, for which we will need the following definitions and identities.

\begin{definition}
The \textnormal{Adjoint action} $Ad$ of $g\in G$ on the vector field $\mathbf{v}_j=\sum_{\alpha,i} (\xi^i_j\partial_{x_i}+\phi^\alpha_j\partial_{u^\alpha})$ is given as follows
$$Ad_g\left(\sum_{\alpha,i}(\xi^i_j\partial_{x_i}+\phi^\alpha_j\partial_{u^\alpha})\right)=\sum_{\alpha,i}(\xi^i_j(\widetilde{\mathbf{x}},\widetilde{\mathbf{u}})\partial_{\widetilde{x_i}}+\phi^\alpha_j(\widetilde{\mathbf{x}},\widetilde{\mathbf{u}})\partial_{\widetilde{u^\alpha}}),$$
so that
\begin{equation}\label{definitionxiphis}
\left(\begin{array}{cc}Ad(\Xi_j) & Ad(\Phi_j)\end{array}\right)=\left(\begin{array}{cc} \Xi_j(\widetilde{\mathbf{x}},\widetilde{\mathbf{u}}) & \Phi_j(\widetilde{\mathbf{x}},\widetilde{\mathbf{u}})\end{array}\right)\left(\dfrac{\partial(\widetilde{\mathbf{x}},\widetilde{\mathbf{u}})}{\partial(\mathbf{x},\mathbf{u})}\right)^{-T},
\end{equation}
with $\Xi_j=(\xi^1_j,...,\xi^p_j)$ and $\Phi_j=(\phi^1_j,...,\phi^q_j)$, and for all $\mathbf{v}_j$, by Theorem 3.3.10 of \cite{Mansfield}, we have that
\begin{equation}
\mathcal{A}d(g)\left(\begin{array}{cc}\Xi(\mathbf{x},\mathbf{u}) & \Phi(\mathbf{x},\mathbf{u})\end{array}\right)=\left(\begin{array}{cc} \Xi(\widetilde{\mathbf{x}},\widetilde{\mathbf{u}}) & \Phi(\widetilde{\mathbf{x}},\widetilde{\mathbf{u}})\end{array}\right)\left(\dfrac{\partial(\widetilde{\mathbf{x}},\widetilde{\mathbf{u}})}{\partial(\mathbf{x},\mathbf{u})}\right)^{-T},
\end{equation}
where $\mathcal{A}d(g)$ is an $r\times r$ matrix, giving the Adjoint action, depending only on the group parameters, with $r=\dim(G)$.
\end{definition}

\begin{lemma}
Let $\mathbf{x}=(x_1,...,x_p)$ and $\mathbf{u}(\mathbf{x})=(u^1(\mathbf{x}),...,u^q(\mathbf{x}))$. The $q \times p$ matrix $\partial \mathbf{u}/\partial \mathbf{x}$ can be written as
\begin{equation}\label{partialderivative}
\dfrac{\partial \mathbf{u}}{\partial \mathbf{x}}=\left(\dfrac{\partial\widetilde{\mathbf{u}}}{\partial \mathbf{u}}-\dfrac{\mathrm{d}\widetilde{\mathbf{u}}}{\mathrm{d}\widetilde{\mathbf{x}}}\dfrac{\partial \widetilde{\mathbf{x}}}{\partial \mathbf{u}}\right)^{-1}\left(\dfrac{\mathrm{d}\widetilde{\mathbf{u}}}{\mathrm{d}\widetilde{\mathbf{x}}}\dfrac{\partial \widetilde{\mathbf{x}}}{\partial \mathbf{x}}-\dfrac{\partial\widetilde{\mathbf{u}}}{\partial \mathbf{x}}\right).
\end{equation}
\end{lemma}

\begin{proof}
We have 
$$\dfrac{\mathrm{d}\widetilde{\mathbf{u}}}{\mathrm{d}\widetilde{\mathbf{x}}}\dfrac{\mathrm{d}\widetilde{\mathbf{x}}}{\mathrm{d}\mathbf{x}}=\dfrac{\mathrm{d}\widetilde{\mathbf{u}}}{\mathrm{d}\mathbf{x}}$$
and
$$\dfrac{\mathrm{d}\widetilde{\mathbf{z}}}{\mathrm{d}\mathbf{x}}=\dfrac{\partial \widetilde{\mathbf{z}}}{\partial \mathbf{x}}+\dfrac{\partial \widetilde{\mathbf{z}}}{\partial \mathbf{u}}\dfrac{\partial \mathbf{u}}{\partial \mathbf{x}},\qquad \mathbf{z}=\mathbf{x},\mathbf{u}.$$
The result follows from expanding the first equation, and collecting terms in $\partial\mathbf{u}/\partial\mathbf{x}$. $\hfill \Box$
\end{proof}

\begin{definition}\label{character}
Given the vector field $\mathbf{v}_j=\sum_{\alpha,i}(\xi^i_j\partial_{x_i}+\phi^\alpha_j\partial_{u^\alpha})$, the column vector $\mathbf{Q}_j$ with components
$$Q^\alpha_j(\mathbf{x},\mathbf{u},\mathbf{u_x})=\phi^\alpha_j(\mathbf{x},\mathbf{u})-\sum_{i=1}^pu^\alpha_i\xi^i_j(\mathbf{x},\mathbf{u}),\quad \alpha=1,...,q,$$
is referred to as the \textnormal{characteristic} of the vector field $\mathbf{v}_j$.
\end{definition}

Letting $g\in G$ act on $\mathbf{Q}_j$, we have
$$\mathbf{Q}_j(\widetilde{\mathbf{x}},\widetilde{\mathbf{u}},\widetilde{\mathbf{u_x}})=\left(\begin{array}{cc}-\dfrac{\mathrm{d}\widetilde{\mathbf{u}}}{\mathrm{d}\widetilde{\mathbf{x}}} & I_q\end{array}\right)\left(\begin{array}{c}\Xi^T_j(\widetilde{\mathbf{x}},\widetilde{\mathbf{u}}) \\ \Phi^T_j(\widetilde{\mathbf{x}},\widetilde{\mathbf{u}})\end{array}\right).$$
Using (\ref{definitionxiphis}) and (\ref{partialderivative}) this can be written as
\begin{eqnarray}\nonumber
\mathbf{Q}_j(\widetilde{\mathbf{x}},\widetilde{\mathbf{u}},\widetilde{\mathbf{u_x}})&\kern-8pt=&\kern-8pt\left(\dfrac{\partial\widetilde{\mathbf{u}}}{\partial \mathbf{u}}-\dfrac{\mathrm{d}\widetilde{\mathbf{u}}}{\mathrm{d}\widetilde{\mathbf{x}}}\dfrac{\partial \widetilde{\mathbf{x}}}{\partial \mathbf{u}}\right)\left(Ad(\Phi^T_j)-\dfrac{\partial \mathbf{u}}{\partial \mathbf{x}}Ad(\Xi^T_j)\right) \\\label{actionQ}
&\kern-8pt=&\kern-8pt \left(\dfrac{\partial\widetilde{\mathbf{u}}}{\partial \mathbf{u}}-\dfrac{\mathrm{d}\widetilde{\mathbf{u}}}{\mathrm{d}\widetilde{\mathbf{x}}}\dfrac{\partial \widetilde{\mathbf{x}}}{\partial \mathbf{u}}\right)
Ad(\mathbf{Q}_j)
\end{eqnarray}
where this defines
\begin{equation}\label{AdQdefn}
Ad(\mathbf{Q}_j)=Ad(\Phi^T_j)-\dfrac{\partial \mathbf{u}}{\partial \mathbf{x}}Ad(\Xi^T_j).
\end{equation}

The following lemma provides a result on the action of an element $g\in G$ on the $p-1$-forms, which will be needed to determine the action on Noether's conservation laws.

\begin{lemma}\label{lemmaaboutcoeff}
If 
$$(-1)^{k-1} \mathrm{d}\widetilde{x_1}...\widehat{\mathrm{d}\widetilde{x_k}}...\mathrm{d}\widetilde{x_p}=\sum_{\ell=1}^p (-1)^{\ell-1} Z^k_\ell \mathrm{d}x_1...\widehat{\mathrm{d}x_\ell}...\mathrm{d}x_p$$
defines $Z^k_\ell$, then
\begin{equation}
(-1)^{\ell-1} Z^k_\ell=\left(\left(\dfrac{\mathrm{d}\widetilde{\mathbf{x}}}{\mathrm{d}\mathbf{x}}\right)^{-1}\right)_{\ell k}\det\left(\dfrac{\mathrm{d}\widetilde{\mathbf{x}}}{\mathrm{d}\mathbf{x}}\right).
\end{equation}
\end{lemma}

The proof of this lemma can be found in Appendix \ref{apendice}.

\begin{theorem}\label{adjointactioncons}
Let $\mathscr{L}[\mathbf{u}]=\int_\Omega L(\mathbf{x},\mathbf{u},u^\alpha_\mathrm{K})\mathrm{d}^p\mathrm{x}$ be a variational problem, which is invariant under the action of a Lie group symmetry $G$ given by
$$\begin{array}{ccccc}
\mathbf{x} & \mapsto  & \kern-8pt g\cdot \mathbf{x} & = & \kern-67pt \widetilde{\mathbf{x}}(\mathbf{x},\mathbf{u}),\\\nonumber
\mathbf{u} & \mapsto & \kern-8pt g\cdot\mathbf{u} & = &\kern-67pt \widetilde{\mathbf{u}}(\mathbf{x},\mathbf{u}),\\
u^\alpha_\mathrm{K} & \mapsto & g\cdot u^\alpha_\mathrm{K} & = & \widetilde{u^\alpha_\mathrm{K}}:=\dfrac{\partial^{|\mathrm{K}|}\widetilde{u^\alpha}}{\partial \widetilde{x_{k_1}}...\partial \widetilde{x_{k_m}}},
\end{array}$$
so that
$$L(\mathbf{x},\mathbf{u},u^\alpha_\mathrm{K})=L(\widetilde{\mathbf{x}},\widetilde{\mathbf{u}},\widetilde{u^\alpha_\mathrm{K}})\det\left(\dfrac{\mathrm{d}\widetilde{\mathbf{x}}}{\mathrm{d}\mathbf{x}}\right).$$
If
$$\sum_{k=1}^p (-1)^{k-1} C^j_k(\mathbf{x},\mathbf{u},u^\alpha_K,\Xi_j(\mathbf{x},\mathbf{u}),\Phi_j(\mathbf{x},\mathbf{u}))\mathrm{d}x_1...\widehat{\mathrm{d}x_k}...\mathrm{d}x_p, \quad\textrm{for }j=1,...,r,$$ are Noether's conservation laws, with $\Xi_j=(\xi^1_j,...,\xi^p_j)$ and $\Phi_j=(\phi^1_j,...,\phi^q_j)$ being the infinitesimals as defined in (\ref{infinitesimals}), then for all $g\in G$
\begin{align}\nonumber
&\sum_{k=1}^p (-1)^{k-1} C^j_k(\widetilde{\mathbf{x}},\widetilde{\mathbf{u}},\widetilde{u^\alpha_K},\Xi_j(\widetilde{\mathbf{x}},\widetilde{\mathbf{u}}),\Phi_j(\widetilde{\mathbf{x}},\widetilde{\mathbf{u}}))\mathrm{d}\widetilde{x_1}...\widehat{\mathrm{d}\widetilde{x_k}}...\mathrm{d}\widetilde{x_p}\\\nonumber
&=\sum_{k=1}^p (-1)^{k-1} C^j_k(\mathbf{x},\mathbf{u},u^\alpha_K,Ad(\Xi^T_j),Ad(\Phi^T_j))\mathrm{d}x_1...\widehat{\mathrm{d}x_k}...\mathrm{d}x_p.
\end{align}
\end{theorem}

To simplify the proof of Theorem \ref{adjointactioncons}, we shall need the following lemma.

\begin{lemma}\label{firstorderlagrangians}
It is sufficient to demonstrate Theorem \ref{adjointactioncons} for a first order Lagrangian with a Lie group symmetry. That is, any Lagrangian invariant under an action of a Lie group $G$ is equivalent to a first order Lagrangian that is also invariant under an extended action of $G$.
\end{lemma}

\begin{proof}
Any Lagrangian can be written as a first order Lagrangian by introducing Lagrangian multipliers and a new dependent variable, $v^\alpha_\mathrm{K}$ for every derivative of $u^\alpha$ appearing as an argument of $L$. Specifically, define
$$\bar{L}=L(\mathbf{x},\mathbf{u},v^\alpha_\mathrm{K},(v^\alpha_\mathrm{J})_\ell)-\sum_{\alpha,\ell} \lambda^\alpha_\ell(u^\alpha_\ell-v^\alpha_\ell)- \sum_{\alpha,\ell,|\mathrm{K}|>0}\lambda^\alpha_{\mathrm{K}\ell} ((v^\alpha_\mathrm{K})_\ell-v^\alpha_{\mathrm{K}\ell}),$$
where $\mathrm{K}=(k_1,...,k_\mathrm{N})$ is an ordered multi-index of differentiation which is at most equal to $\mathrm{J}=(j_1,...,j_\mathrm{n})$. The Euler-Lagrange equations for $\bar{L}$ are
\begin{align}\nonumber
&\mathsf{E}^u(\bar{L})=\left\{\dfrac{\partial L}{\partial u^\alpha}+\sum_{i=1}^pD_i(\lambda^\alpha_i)\;\Big|\;\alpha \right\},\\[10pt]\nonumber
&\mathsf{E}^v(\bar{L})=\left\{\dfrac{\partial L}{\partial v^\alpha_\mathrm{K}}+\lambda^\alpha_\mathrm{K}+\sum_{\ell\ge k_\mathrm{N}}D_\ell(\lambda^\alpha_{\mathrm{K}\ell}) \;\Big|\;\alpha,\mathrm{K}\right\}\cup\left\{\dfrac{\partial L}{\partial v^\alpha_\mathrm{J}}-\sum_{\ell\ge j_n}D_\ell\left(\dfrac{\partial L}{\partial (v^\alpha_\mathrm{J})_\ell}\right)+\lambda^\alpha_\mathrm{J}\;\Big|\;\alpha,\mathrm{J}\right\},\\[10pt]\nonumber
&\mathsf{E}^\lambda(\bar{L})=\left\{u^\alpha_\ell-v^\alpha_\ell\;|\;\alpha,\ell\right\}\cup\left\{(v^\alpha_\mathrm{K})_\ell-v^\alpha_{\mathrm{K}\ell} \;|\;\alpha,\mathrm{K},\ell\right\}.
\end{align}
Eliminating the $v$'s and the $\lambda$'s yields the Euler-Lagrange system for $L$. We now induce an action on the additional dependent variables as follows. Set
$$\begin{array}{rcl}
g\cdot v^\alpha_\mathrm{K}&=&(g\cdot u^\alpha_\mathrm{K})|_{\{u^\alpha_\mathrm{M}=v^\alpha_\mathrm{M}\;|\;|\mathrm{M}|>0\}},\\[10pt]
g\cdot\lambda^\alpha_\ell&=& \left(\left(\dfrac{g\cdot u^\alpha_\ell-g\cdot v^\alpha_\ell}{u^\alpha_\ell-v^\alpha_\ell}\right)\det \left(\dfrac{\mathrm{d}(g\cdot \mathbf{x})}{\mathrm{d}\mathbf{x}}\right)\right)^{-1}\lambda^\alpha_\ell,\\[10pt]
g\cdot\lambda^\alpha_{\mathrm{K}\ell}&=& \left(\left(\dfrac{g\cdot (v^\alpha_\mathrm{K})_\ell-g\cdot v^\alpha_{\mathrm{K}\ell}}{(v^\alpha_\mathrm{K})_\ell-v^\alpha_{\mathrm{K}\ell}}\right)\det \left(\dfrac{\mathrm{d}(g\cdot \mathbf{x})}{\mathrm{d}\mathbf{x}}\right)\right)^{-1}\lambda^\alpha_{\mathrm{K}\ell},
\end{array}$$
and thus, by construction $\bar{L}\mathrm{d}^p\mathrm{x}$ is invariant. This is indeed a group action: the action on the $v^\alpha_\mathrm{K}$ is symbolically that of the action on the derivatives, $u^\alpha_\mathrm{K}$, which is a right action. Further,
$$\begin{array}{rcl}
h\cdot(g\cdot \lambda^\alpha_\ell) &= &h\cdot\left(\left( \dfrac{g\cdot u^\alpha_\ell-g\cdot v^\alpha_\ell}{u^\alpha_\ell-v^\alpha_\ell}\right)\det\left(\dfrac{\mathrm{d}(g\cdot\mathbf{x})}{\mathrm{d}\mathbf{x}}\right)\right)^{-1}\lambda^\alpha_\ell\\[10pt]
&=&\left(\left( \dfrac{gh\cdot u^\alpha_\ell-gh\cdot v^\alpha_\ell}{h\cdot u^\alpha_\ell-h\cdot v^\alpha_\ell}\right)\det\left(\dfrac{\mathrm{d}(gh\cdot\mathbf{x})}{\mathrm{d}(h\cdot\mathbf{x})}\right)\right)^{-1}h\cdot\lambda^\alpha_\ell\\[10pt]
&=&\left(\left( \dfrac{gh\cdot u^\alpha_\ell-gh\cdot v^\alpha_\ell}{u^\alpha_\ell-v^\alpha_\ell}\right)\det\left(\dfrac{\mathrm{d}(gh\cdot\mathbf{x})}{\mathrm{d}\mathbf{x}}\right)\right)^{-1}\lambda^\alpha_\ell\\[10pt]
&=&gh\cdot \lambda^\alpha_\ell
\end{array}$$
by the chain rule and using the fact that the determinant is multiplicative.

The argument for $\lambda^\alpha_\mathrm{K}$ is similar. Finally, we note that obtaining Noether's conservation laws for $\bar{L}$ and eliminating the $v^\alpha_\mathrm{K}$ and $\lambda^\alpha_\mathrm{K}$ using the Euler-Lagrange equations $\mathsf{E}^v(\bar{L})$ and $\mathsf{E}^\lambda(\bar{L})$, yields the conservation laws for $L$. $\hfill \Box$
\end{proof}

\begin{proof} \textbf{of Theorem \ref{adjointactioncons}} By Lemma \ref{firstorderlagrangians}, it is enough to prove the result for a first order Lagrangian. A first order Lagrangian with a Lie symmetry has Noether's conservation laws in the form
$$\sum_{k=1}^p \dfrac{\mathrm{d}}{\mathrm{d}x_k}C^j_k=0,\quad \textrm{for } j=1,...,r,$$
where 
$$C^j_k=L(\mathbf{x},\mathbf{u},u^\alpha_K)\xi^k_j(\mathbf{x},\mathbf{u})+\sum_{\alpha=1}^q Q^\alpha_j(\mathbf{x},\mathbf{u},\mathbf{u_x})\dfrac{\partial L}{\partial u^\alpha_k}$$ 
and $Q^\alpha_j$ is as defined in Definition \ref{character}.

\noindent\textbf{Step 1} Now considering the operator used for the $k^{\mathrm{th}}$ component of the conservation law
$$\sum_{\alpha=1}^q Q^\alpha_j(\mathbf{x},\mathbf{u},\mathbf{u_x})\dfrac{\partial}{\partial u^\alpha_k}$$
where $k$ is fixed, we will show that the action of $g\in G$ on the operator is equal to
$$\sum_{\alpha=1}^q Q^\alpha_j(\widetilde{\mathbf{x}},\widetilde{\mathbf{u}},\widetilde{\mathbf{u_x}})\dfrac{\partial}{\partial \widetilde{u^\alpha_k}}=\sum_{\alpha,\ell}Ad(Q^\alpha_j)\left(\dfrac{\mathrm{d}\widetilde{\mathbf{x}}}{\mathrm{d}\mathbf{x}}\right)_{k\ell}\dfrac{\partial}{\partial u^\alpha_\ell}.$$
Since we know what the action of $g\in G$ is on $\mathbf{Q}_j$ (see (\ref{actionQ})), we just need to find how $g\in G$ acts on $\partial/\partial u^\alpha_k$. Schematically, we have that
$$\boldsymbol{\nabla}_{\mathbf{\widetilde{u_x}}}=\left(\dfrac{\mathrm{d}\widetilde{\mathbf{u_x}}}{\mathrm{d}\mathbf{u_x}}\right)^{-T}\boldsymbol{\nabla}_{\mathbf{u_x}},$$
and to obtain the components of this Jacobian matrix, we consider Equation (\ref{partialderivative}) and calculate
$$\lim_{\varepsilon\rightarrow 0}\left.\dfrac{\partial \mathbf{u}}{\partial \mathbf{x}}\right|_{\displaystyle{\frac{\mathrm{d}\widetilde{\mathbf{u}}}{\mathrm{d}\widetilde{\mathbf{x}}}+\varepsilon H}}=\left(\dfrac{\partial\widetilde{\mathbf{u}}}{\partial\mathbf{u}}-\dfrac{\mathrm{d}\widetilde{\mathbf{u}}}{\mathrm{d}\widetilde{\mathbf{x}}}\dfrac{\partial \widetilde{\mathbf{x}}}{\partial \mathbf{u}}\right)^{-1}H\dfrac{\mathrm{d}\widetilde{\mathbf{x}}}{\mathrm{d}\mathbf{x}}=A^{-1}HB=V(H),$$
where this defines $A$, $B$ and $V(H)$. By construction, the coefficient of $H_{\alpha k}$ in the $(\beta,\ell)$ component of this matrix equals
$$\dfrac{\partial u^\beta_\ell}{\partial \widetilde{u^\alpha_k}}.$$
Direct calculation shows that if $\mathbf{e_{ij}}$ is the matrix with $(\mathbf{e_{ij}})_{k\ell}=\delta_{ik}\delta_{j\ell}$, then
$$V(\mathbf{e_{ij}})=\left(\begin{array}{c}
(A^{-1})_{1i}\\
(A^{-1})_{2i}\\
\vdots\\
(A^{-1})_{qi}\end{array}\right)\left(\begin{array}{cccc} 
B_{j1} & B_{j2} & \cdots & B_{jp}\end{array}\right),$$
and thus
$$\dfrac{\partial u^\beta_\ell}{\partial \widetilde{u^\alpha_k}}=\left(\left(\dfrac{\partial\widetilde{\mathbf{u}}}{\partial\mathbf{u}}-\dfrac{\mathrm{d}\widetilde{\mathbf{u}}}{\mathrm{d}\widetilde{\mathbf{x}}}\dfrac{\partial \widetilde{\mathbf{x}}}{\partial \mathbf{u}}\right)^{-1}\right)_{\beta\alpha}\left(\dfrac{\mathrm{d}\widetilde{\mathbf{x}}}{\mathrm{d}\mathbf{x}}\right)_{k\ell}.$$
We have then, for $k$ fixed,
\begin{align}\nonumber
& \sum_{\alpha=1}^q Q^\alpha_j(\widetilde{\mathbf{x}},\widetilde{\mathbf{u}},\widetilde{\mathbf{u_x}})\dfrac{\partial}{\partial \widetilde{u^\alpha_k}}\\\nonumber
& \quad = \sum_{\beta,\ell,n,\alpha} A_{\alpha n}Ad(Q^n_j) (A^{-1})_{\beta\alpha}B_{k\ell}\dfrac{\partial}{\partial u^\beta_\ell}\\\nonumber
& \quad = \sum_{\beta,\ell} Ad(Q^\beta_j)\left(\dfrac{\mathrm{d}\widetilde{\mathbf{x}}}{\mathrm{d}\mathbf{x}}\right)_{k\ell}\dfrac{\partial}{\partial u^\beta_\ell},
\end{align}
using (\ref{actionQ}), and noting that the matrix appearing as a factor of $Q(\widetilde{\mathbf{x}},\widetilde{\mathbf{u}},\widetilde{\mathbf{u_x}})$ is $A$.

\noindent\textbf{Step 2} Now we evaluate $\sum_\alpha Q^\alpha_j(\widetilde{\mathbf{x}},\widetilde{\mathbf{u}},\widetilde{\mathbf{u_x}})\partial / \partial \widetilde{u^\alpha_k}$ on
\begin{equation}\label{invariancecond}
L(\widetilde{\mathbf{x}},\widetilde{\mathbf{u}},\widetilde{\mathbf{u_x}})=L(\mathbf{x},\mathbf{u},\mathbf{u_x})\det\left(\dfrac{\mathrm{d}\widetilde{\mathbf{x}}}{\mathrm{d}\mathbf{x}}\right)^{-1},
\end{equation}
which is the invariance condition on the Lagrangian. From
$$\dfrac{\mathrm{d}\widetilde{\mathbf{x}}}{\mathrm{d}\mathbf{x}}=\dfrac{\partial \widetilde{\mathbf{x}}}{\partial \mathbf{x}}+\dfrac{\partial \widetilde{\mathbf{x}}}{\partial \mathbf{u}}\dfrac{\partial \mathbf{u}}{\partial \mathbf{x}}
$$
it can be shown that
$$\begin{array}{rl}
\dfrac{\partial}{\partial u^\beta_\ell}\det\left(\dfrac{\mathrm{d}\widetilde{\mathbf{x}}}{\mathrm{d}\mathbf{x}}\right) \kern-6pt&=\displaystyle{\sum_{j=1}^p\dfrac{\partial \widetilde{x_j}}{\partial u^\beta}\left((j,\ell)\textrm{ first minor of } \dfrac{\mathrm{d}\widetilde{\mathbf{x}}} {\mathrm{d}\mathbf{x}}\cdot (-1)^{j+\ell}\right)}\\[10pt]
& = \displaystyle{\sum_{j=1}^p}\dfrac{\partial \widetilde{x_j}}{\partial u^\beta}\left(\left(\dfrac{\mathrm{d}\widetilde{\mathbf{x}}} {\mathrm{d}\mathbf{x}}\right)^{-1}\right)_{\ell j}\det\left(\dfrac{\mathrm{d}\widetilde{\mathbf{x}}} {\mathrm{d}\mathbf{x}}\right).
\end{array}$$
Thus, we obtain, recalling $k$ is fixed, that
\begin{align}\nonumber
& \sum_{\alpha=1}^q Q^\alpha_j(\widetilde{\mathbf{x}},\widetilde{\mathbf{u}},\widetilde{\mathbf{u_x}})\dfrac{\partial}{\partial \widetilde{u^\alpha_k}}\left(L(\widetilde{\mathbf{x}},\widetilde{\mathbf{u}},\widetilde{\mathbf{u_x}})\right)\\\label{operatoronL}
&\quad = \det\left(\dfrac{\mathrm{d}\widetilde{\mathbf{x}}} {\mathrm{d}\mathbf{x}}\right)^{-1}\left(\sum_{\beta,\ell} Ad(Q^\beta_j)\left(\dfrac{\mathrm{d}\widetilde{\mathbf{x}}} {\mathrm{d}\mathbf{x}}\right)_{k\ell}\dfrac{\partial}{\partial u^\beta_\ell}L(\mathbf{x},\mathbf{u},\mathbf{u_x})-\sum_{\beta}Ad(Q^\beta_j)\dfrac{\partial \widetilde{x_k}}{\partial u^\beta}L(\mathbf{x},\mathbf{u},\mathbf{u_x})\right).
\end{align}
\noindent\textbf{Step 3} We are now in a position to consider the $k^{\mathrm{th}}$ component of the conservation law in the transformed variables, namely,
$$g\cdot C^j_k=L(\widetilde{\mathbf{x}},\widetilde{\mathbf{u}},\widetilde{\mathbf{u_x}})\xi^k_j(\widetilde{\mathbf{x}},\widetilde{\mathbf{u}})+\sum_\alpha Q^\alpha_j(\widetilde{\mathbf{x}},\widetilde{\mathbf{u}},\widetilde{\mathbf{u_x}})\dfrac{\partial}{\partial \widetilde{u^\alpha_k}}L(\widetilde{\mathbf{x}},\widetilde{\mathbf{u}},\widetilde{\mathbf{u_x}}).
$$
Using Equations (\ref{definitionxiphis}), (\ref{invariancecond}) and (\ref{operatoronL}), and collecting terms, yields
\begin{equation}\label{penultimatestep}
g\cdot C^j_k=\det\left(\dfrac{\mathrm{d}\widetilde{\mathbf{x}}} {\mathrm{d}\mathbf{x}}\right)^{-1}\left(\dfrac{\mathrm{d}\widetilde{\mathbf{x}}} {\mathrm{d}\mathbf{x}}\right)_{k\ell}\left(L(\mathbf{x},\mathbf{u},\mathbf{u_x})Ad(\xi^k_j)+\sum_\alpha Ad(Q^\alpha_j)\dfrac{\partial}{\partial u^\alpha_\ell}L(\mathbf{x},\mathbf{u},\mathbf{u_x})\right).
\end{equation}
\noindent\textbf{Step 4} We now consider
$$g\cdot \left(\sum_{k=1}^p(-1)^{k-1}C^j_k\mathrm{d}x_1...\,\widehat{\mathrm{d}x_k}...\,\mathrm{d}x_p\right)=\sum_{k=1}^p(-1)^{k-1}(g\cdot C^j_k)\mathrm{d}\widetilde{x_1}...\,\widehat{\mathrm{d}\widetilde{x_k}}...\,\mathrm{d}\widetilde{x_p},$$
Combining Equation (\ref{penultimatestep}) and Lemma \ref{lemmaaboutcoeff} yields
\begin{align}\nonumber
&g\cdot\left(\sum_{k=1}^p (-1)^{k-1} C^j_k(\mathbf{x},\mathbf{u},\mathbf{u_x},\Xi_j,\Phi_j)\mathrm{d}x_1...\widehat{\mathrm{d}x_k}...\mathrm{d}x_p\right)\\
&\quad=\sum_{k=1}^p (-1)^{k-1} C^j_k(\mathbf{x},\mathbf{u},\mathbf{u_x},Ad(\Xi_j),Ad(\Phi_j))\mathrm{d}x_1...\widehat{\mathrm{d}x_k}...\mathrm{d}x_p,
\end{align}
which completes the proof. $\hfill \Box$
\end{proof}

{Since we can write the Adjoint action on the generating vector fields in matrix form (see (\ref{definitionxiphis})) and the conservation laws are linear in $\xi$ and $\phi$, the action of $g\in G$ on the conservation laws can be written as
\begin{equation}
\mathcal{A}d(g)\left(\begin{array}{c}
\displaystyle{\sum_{k=1}^p (-1)^{k-1}C^1_k }\\
\vdots\\
\displaystyle{\sum_{k=1}^p (-1)^{k-1}C^r_k}\end{array}\right),
\end{equation}
where $\mathcal{A}d(g)$ is the Adjoint representation of $G$ which can be easily computed as shown in the following example. 
\begin{example}\label{SL2adjointrep}
Consider the infinitesimal vector fields 
$$x\partial_x-y\partial_y,\quad y\partial_x\quad and\quad x\partial_y,$$
which generate the linear $SL(2)$ action. The adjoint action of $g\in SL(2)$ on these infinitesimal vector fields is as follows
\begin{align}\nonumber
& g\cdot \left(\alpha(x\partial_x-y\partial_y)+\beta y\partial_x+\gamma x\partial_y\right)\\\nonumber
& \quad=\alpha(\widetilde{x}\partial_{\widetilde{x}}-\widetilde{y}\partial_{\widetilde{y}})+\beta \widetilde{y}\partial_{\widetilde{x}}+\gamma \widetilde{x}\partial_{\widetilde{y}}\\\label{Adjointrepresentation}
&\quad =\big(\begin{array}{ccc}\alpha & \beta & \gamma \end{array}\big)\underbrace{\left(\begin{array}{ccc}
ad+bc & 2bd & -2ac\\
cd & d^2 & -c^2\\
-ab & -b^2 & a^2\end{array}\right)}_{\mathlarger{\mathcal{A}d(g)}}\left(\begin{array}{c}
x\partial_x-y\partial_y\\
y\partial_x\\
x\partial_y\end{array}\right),
\end{align}
where $ad-bc=1$.
\end{example}
For more details on the adjoint representation of $G$ with respect to the generating vector fields, see Gon\c calves and Mansfield \cite{Mansfield,GoncalvesMansfield}. }

\subsection{Noether's Laws in terms of the invariants and the Adjoint action of a moving frame}
The following result states the structure of Noether's conservation laws for the general case, where the independent variables are not necessarily invariant under the Lie group action.

\begin{theorem}\label{conslawsnoninvindvar}
Let $\int L(\kappa_1,\kappa_2,...)I(\mathrm{d}^p\mathbf{x})$ be invariant under $G \times M\rightarrow M$, where $M=J^n(X\times U)$, with generating invariants $\kappa_j$, for $j=1,...,N$. Introduce a dummy invariant variable $\tau$ to effect the variation and then integration by parts yields
$$\begin{array}{l}
\mathcal{D}_\tau \int L(\kappa_1,\kappa_2,...) I(\mathrm{d}^p\mathbf{x})\\
\displaystyle{\;=\int \left[\sum_\alpha\mathsf{E}^\alpha(L)I^\alpha_\tau I(\mathrm{d}^p\mathbf{x})+\sum_{k=1}^p \mathrm{d}\left((-1)^{k-1}\left(\sum_{\mathrm{J},\alpha} I^\alpha_{\mathrm{J}\tau}C^\alpha_{\mathrm{J},k}\right) I(\mathrm{d}x_1)...\widehat{I(\mathrm{d}x_k)}...I(\mathrm{d}x_p)\right)\right],}
\end{array}$$
where this defines the vectors $\mathcal{C}^\alpha_k=(C^\alpha_{\mathrm{J},k})$. Recall that $\mathsf{E}^\alpha(L)$ are the invariantized Euler-Lagrange equations and $I^\alpha_{\mathrm{J}\tau}=I(u^\alpha_{\mathrm{J}\tau})$, where $\mathrm{J}$ is a multi-index of differentiation with respect to the variables $x_i$, for $i=1,...,p$. Let $(a_1,...,a_r)$ be the coordinates of $G$ near the identity $e$, and $\mathbf{v}_i$, for $i=1,...,r$, the associated infinitesimal vector fields. Furthermore, let $\mathcal{A}d(g)$ be the Adjoint representation of $G$ with respect to these vector fields. For each dependent variable, define the matrices of characteristics to be
$$\mathscr{Q}^\alpha(\widetilde{\zede})=(\widetilde{D_\mathrm{K}(Q^\alpha_i)} ),\qquad \alpha=1,...,q,$$
where $\mathrm{K}$ is a multi-index of differentiation with respect to the $x_k$ and 
$$Q^\alpha_i=\phi^\alpha_i-\sum_{k=1}^p \xi^k_i u^\alpha_k=\left.\frac{\partial \widetilde{u^\alpha}}{\partial a_i}\right|_{g=e}-\sum_{k=1}^p \left.\frac{\partial \widetilde{x_k}}{\partial a_i}\right|_{g=e}u^\alpha_k$$
are the components of the $q$-tuple $\mathbf{Q}_i$ known as the characteristic of the vector field $\mathbf{v}_i$. Let $\mathscr{Q}^\alpha(J,I)$, for $\alpha=1,...,q$, be the invariantization of the above matrices. Then, the $r$ conservation laws obtained via Noether's Theorem can be written in the form
$$\mathrm{d} \left(\mathcal{A}d(\rho)^{-1}\left(\boldsymbol{\upsilon}_1,\cdots,\boldsymbol{\upsilon}_p\right)\mathsf{M}_{\mathcal{J}}\,\mathrm{d}^{p-1}\widehat{\mathbf{x}}\,\right)=0,$$
where 
\begin{equation}\label{vectorofinvs}
\boldsymbol{\upsilon}_k=\sum_\alpha (-1)^{k-1}\left(\mathscr{Q}^\alpha(J,I)\mathcal{C}^\alpha_k+L(\Xi(J,I))_{k}\right),
\end{equation}
are the vectors of invariants, with $(\Xi(J,I))_k$ the $k^\textrm{th}$ column of $\Xi(J,I)$, $\mathsf{M}_{\mathcal{J}}$ is the matrix of first minors of the Jacobian matrix evaluated at the frame, $\mathcal{J}=\mathrm{d}\widetilde{\mathbf{x}}/\mathrm{d}\mathbf{x}|_{g=\rho(\zede)}$, and 
$$\mathrm{d}^{p-1}\widehat{\mathbf{x}}=\left(\begin{array}{c}
\widehat{\mathrm{d}x_1}\mathrm{d}x_2...\mathrm{d}x_p\\
\mathrm{d}x_1\widehat{\mathrm{d}x_2}\mathrm{d}x_3...\mathrm{d}x_p\\
\vdots\\
\mathrm{d}x_1...\mathrm{d}x_{p-1}\widehat{\mathrm{d}x_p}
\end{array}\right)=\left(\begin{array}{c}\mathrm{d}x_2\mathrm{d}x_3...\mathrm{d}x_p\\
\mathrm{d}x_1\mathrm{d}x_3...\mathrm{d}x_p\\
\vdots\\
\mathrm{d}x_1\mathrm{d}x_2...\mathrm{d}x_{p-1}\end{array}\right).$$
\end{theorem}

\begin{proof}
The infinitesimal criterion of invariance tells us that $G$ is a variational symmetry group of $\int \bar{L}(\zede)\mathrm{d}^p\mathbf{x}$ if and only if 
$$\mathsf{pr}^{(n)}\mathbf{v}_i(\bar{L})+\bar{L}\mathsf{Div}\,\Xi_i=0,$$
for all $\zede\in M$ and every infinitesimal generator $\mathbf{v}_i$; the $n^{\mathrm{th}}$ prolongation of $\mathbf{v}_i$ is defined as $\mathsf{pr}^{(n)}\mathbf{v}_i=\sum_k\xi^k_i\partial_{x_k}+\sum_{\alpha,\mathrm{J}}\phi^\alpha_{\mathrm{J},i}\partial_{u^\alpha_\mathrm{J}}$. This criterion can also be written as
$$\mathsf{pr}^{(n)}\mathbf{v}_{\mathbf{Q}_i}(\bar{L})+\mathsf{Div}(\bar{L}\Xi_i)=0,$$
where $\mathsf{pr}^{(n)}\mathbf{v}_{\mathbf{Q}_i}=\sum_{\alpha,\mathrm{J}}D_\mathrm{J}Q^\alpha_i\partial_{u^\alpha_\mathrm{J}}$.
Calculating $\int \mathsf{pr}^{(n)}\mathbf{v}_{\mathbf{Q}_i}(\bar{L})\mathrm{d}^p\mathbf{x}$ yields 
$$\int \left(\mathbf{Q}_i\cdot \mathsf{E}(\bar{L})+\mathsf{Div}(\mathrm{A})\right) \mathrm{d}^p\mathbf{x},$$
which is exactly what $\mathrm{d}/\mathrm{d}\varepsilon|_{\varepsilon=0}\bar{\mathscr{L}}[u^\alpha+\varepsilon v^\alpha]$ produces, where $v^\alpha$ correspond to the infinitesimals. Since we know that
$$\left.\frac{\mathrm{d}}{\mathrm{d}\varepsilon}\right|_{\varepsilon=0}\bar{\mathscr{L}}[u^\alpha+\varepsilon v^\alpha]\quad \mathrm{and}\quad \left.\frac{\mathrm{d}}{\mathrm{d}\tau}\right |_{u^\alpha_\tau=v^\alpha}\,\bar{\mathscr{L}}[u^\alpha]$$
yield the same symbolic result,
$$\mathcal{D}_\tau |_{\widetilde{D_\tau} \widetilde{u^\alpha}|_{g=\rho(\zede)}=v^\alpha}\,\mathscr{L}[\boldsymbol{\kappa}]$$
provides us with the invariantized Euler-Lagrange system and the boundary terms
\begin{equation}\label{invboundaryterms}
\sum_{k=1}^p \mathrm{d}\left((-1)^{k-1}\left(\sum_{\mathrm{J},\alpha} I^\alpha_{\mathrm{J}\tau}C^\alpha_{\mathrm{J},k}\right)I(\mathrm{d}x_1)... \widehat{I(\mathrm{d}x_k)}...I(\mathrm{d}x_p)\right).
\end{equation}

By definition, $I^\alpha_{\mathrm{J}\tau}$ is equal to
$$I^\alpha_{\mathrm{J}\tau}=\widetilde{D_\tau}\,\widetilde{u^\alpha_{\mathrm{J}}}|_{g=\rho(\zede)}.$$ 
Hence by the chain rule,
\begin{equation}\label{splitininv}
(I^\alpha_\tau \; I^\alpha_{\mathrm{J}_1 \tau} \; I^\alpha_{\mathrm{J}_2\tau}\; \cdots)=(\widetilde{D_\tau}u^\alpha \; \widetilde{D_\tau}u^\alpha_{\mathrm{J}_1 }\; \widetilde{D_\tau}u^\alpha_{\mathrm{J}_2}\; \cdots)|_{g=\rho(\zede)}\,\left.\frac{\partial(\widetilde{u^\alpha},\widetilde{u^\alpha_{\mathrm{J}_1}},\widetilde{u^\alpha_{\mathrm{J}_2}},...)}{\partial(u^\alpha,u^\alpha_{\mathrm{J}_1},u^\alpha_{\mathrm{J}_2},...)}\right|_{g=\rho(\zede)}^T,
\end{equation}
where the $\mathrm{J}_k$ are multi-indices of differentiation with respect to $x_i$, for $i=1,...,p$. 

We know that the Jacobian matrix $\mathcal{J}=\mathrm{d}\widetilde{\mathbf{x}}/\mathrm{d}\mathbf{x}|_{g=\rho(\zede)}$ can be written as a partitioned matrix
$$\mathcal{J}=\left(\begin{array}{cccc}
\left.\frac{\partial \widetilde{x_1}}{\partial x_1}\right|_{g=\rho(\zede)} & \cdots & \left.\frac{\partial \widetilde{x_1}}{\partial x_p}\right|_{g=\rho(\zede)} & \left.\frac{\partial \widetilde{x_1}}{\partial \tau}\right|_{g=\rho(\zede)}\\
\vdots & \ddots & \vdots & \vdots\\
\left.\frac{\partial \widetilde{x_p}}{\partial x_1}\right|_{g=\rho(\zede)} & \cdots & \left.\frac{\partial \widetilde{x_p}}{\partial x_p}\right|_{g=\rho(\zede)} & \left.\frac{\partial \widetilde{x_p}}{\partial \tau}\right|_{g=\rho(\zede)}\\
\left.\frac{\partial \widetilde{\tau}}{\partial x_1}\right|_{g=\rho(\zede)} & \cdots & \left.\frac{\partial \widetilde{\tau}}{\partial x_p}\right|_{g=\rho(\zede)} & \left.\frac{\partial \widetilde{\tau}}{\partial \tau}\right|_{g=\rho(\zede)}
\end{array}\right)=\left(\begin{array}{cc}
A^T & \mathbf{b}^T\\
\mathbf{0} & 1\end{array}\right),$$ 
where this defines $A$ and $\mathbf{b}$, and that
$$\widetilde{D_\tau}u^\alpha_{\mathrm{J}_\ell}|_{g=\rho(\zede)}=-\mathbf{b}A^{-1}\left(\begin{array}{c}
\partial_{x_1}\\
\vdots\\
\partial_{x_p}\end{array}\right)u^\alpha_{\mathrm{J}_\ell}+\frac{\partial u^\alpha_{\mathrm{J}_\ell}}{\partial \tau}=\frac{\partial u^\alpha_{\mathrm{J}_\ell}}{\partial \tau}-\frac{\partial x_1}{\partial \tau}u^\alpha_{{\mathrm{J}_\ell}1}-\cdots -\frac{\partial x_p}{\partial \tau}u^\alpha_{{\mathrm{J}_\ell}p}.
$$
{Next consider
\begin{equation}\label{correspondence}
\begin{array}{l}
\displaystyle{\left.\frac{\partial \widetilde{u^\alpha}}{\partial \tau}\right|_{g=e}-\left.\frac{\partial \widetilde{x_1}}{\partial \tau}\right|_{g=e}u^\alpha_{1}-\cdots -\left.\frac{\partial \widetilde{x_p}}{\partial \tau}\right|_{g=e}u^\alpha_{p}=u^\alpha_\tau}\\[10pt]
\quad\displaystyle{=Q^\alpha_i=\phi^\alpha_i-\sum_{k=1}^p \xi^k_i u^\alpha_k=\left.\frac{\partial \widetilde{u^\alpha}}{\partial a_i}\right|_{g=e}-\left.\frac{\partial \widetilde{x_1}}{\partial a_i}\right|_{g=e}u^\alpha_{1}-\cdots -\left.\frac{\partial \widetilde{x_p}}{\partial a_i}\right|_{g=e}u^\alpha_{p}},
\end{array}\end{equation}
and
\begin{equation}\label{correspondence1}
\begin{array}{l}
\displaystyle{\left.\frac{\partial \widetilde{u^\alpha_{\mathrm{J}_\ell}}}{\partial \tau}\right|_{g=e}-\left.\frac{\partial \widetilde{x_1}}{\partial \tau}\right|_{g=e}u^\alpha_{\mathrm{J}_\ell 1}-\cdots -\left.\frac{\partial \widetilde{x_p}}{\partial \tau}\right|_{g=e}u^\alpha_{\mathrm{J}_\ell p}=u^\alpha_{\mathrm{J}_\ell\tau}}\\[10pt]
\quad\displaystyle{=D_{\mathrm{J}_\ell}Q^\alpha_i=\phi^\alpha_{\mathrm{J}_\ell,i}-\sum_{k=1}^p \xi^k_i u^\alpha_{\mathrm{J}_\ell k}=\left.\frac{\partial \widetilde{u^\alpha_{\mathrm{J}_\ell}}}{\partial a_i}\right|_{g=e}-\left.\frac{\partial \widetilde{x_1}}{\partial a_i}\right|_{g=e}u^\alpha_{\mathrm{J}_\ell 1}-\cdots -\left.\frac{\partial \widetilde{x_p}}{\partial a_i}\right|_{g=e}u^\alpha_{\mathrm{J}_\ell p}},
\end{array}\end{equation}
so that $\tau$ is considered to be the group parameter, $a_i$.}
 
Furthermore, from Theorem \ref{equivalentto3.3.10} we know that 
\begin{equation}\label{formulaofT3310}
\mathcal{A}d(\rho)^{-1}\mathscr{Q}^\alpha(J,I)=\mathscr{Q}^\alpha(\zede)\left.\left(\frac{\partial \widetilde{u^\alpha_\mathrm{J}}}{\partial u^\alpha_\mathrm{J}}\right)^{T}\right|_{g=\rho(\zede)}
\end{equation}
where $\mathscr{Q}^\alpha(\zede)=(D_\mathrm{K}(Q^\alpha_i))$.

Substituting the vector $(I^\alpha_\tau \; I^\alpha_{\mathrm{J}_1 \tau} \; I^\alpha_{\mathrm{J}_2\tau}\; \cdots)$ in (\ref{invboundaryterms}) by its expression in Equation  (\ref{splitininv}) yields
\footnotesize{$$\sum_{k=1}^p \mathrm{d}\left((-1)^{k-1}\left(\sum_{\alpha} (\widetilde{D_\tau}u^\alpha \; \widetilde{D_\tau}u^\alpha_{\mathrm{J}_1 }\; \widetilde{D_\tau}u^\alpha_{\mathrm{J}_2}\; \cdots)|_{g=\rho(\zede)}\left.\frac{\partial\widetilde{u^\alpha_\mathrm{J}}}{\partial u^\alpha_\mathrm{J}}\right|_{g=\rho(\zede)}^T \mathcal{C}^\alpha_{k}\right)I(\mathrm{d}x_1)\cdots \widehat{I(\mathrm{d}x_k)}\cdots I(\mathrm{d}x_p)\right).$$}\normalsize
By (\ref{correspondence}) and (\ref{correspondence1}), the vector $ (\widetilde{D_\tau}u^\alpha \; \widetilde{D_\tau}u^\alpha_{\mathrm{J}_1 }\; \widetilde{D_\tau}u^\alpha_{\mathrm{J}_2}\; \cdots)$ in the above equation can be substituted by every single row of the matrix of characteristics $\mathscr{Q}^\alpha(\zede)$. Hence, for each independent group parameter $a_i$ we obtain
$$\sum_{k=1}^p \mathrm{d}\left((-1)^{k-1}\left(\sum_{\alpha} \mathscr{Q}^\alpha_i(\zede) \left.\frac{\partial \widetilde{u^\alpha_\mathrm{J}}}{\partial u^\alpha_\mathrm{J}}\right|_{g=\rho(\zede)}^T \mathcal{C}^\alpha_{k}\right)I(\mathrm{d}x_1)\cdots \widehat{I(\mathrm{d}x_k)}\cdots I(\mathrm{d}x_p)\right),\quad i=1,...,r,$$
where $\mathscr{Q}^\alpha_i(\zede)$ corresponds to row $i$ in $\mathscr{Q}^\alpha(\zede)$.

If we have $r$ group parameters describing group elements near the identity of the group, we can write the $r$ equations in matrix form as
$$\sum_{k=1}^p \mathrm{d}\left((-1)^{k-1}\left(\sum_{\alpha} \mathscr{Q}^\alpha(\zede) \left.\frac{\partial \widetilde{u^\alpha_\mathrm{J}}}{\partial u^\alpha_\mathrm{J}}\right|_{g=\rho(\zede)}^T \mathcal{C}^\alpha_{k}\right)I(\mathrm{d}x_1)\cdots \widehat{I(\mathrm{d}x_k)}\cdots I(\mathrm{d}x_p)\right).$$
Using the equality (\ref{formulaofT3310}), we obtain
\begin{equation}\label{almostthere}
\sum_{k=1}^p \mathrm{d}\left((-1)^{k-1}\left(\mathcal{A}d(\rho)^{-1}\sum_\alpha \mathscr{Q}^\alpha(J,I)\mathcal{C}^\alpha_k \right)I(\mathrm{d}x_1)\cdots \widehat{I(\mathrm{d}x_k)}\cdots I(\mathrm{d}x_p) \right).
\end{equation}

Next, it is a standard computation in differential exterior algebra to show that 
\begin{equation}\label{invariantp_1forms}
\left(\begin{array}{c}
\widehat{I(\mathrm{d}x_1)}I(\mathrm{d}x_2)\cdots I(\mathrm{d}x_p)\\
I(\mathrm{d}x_1)\widehat{I(\mathrm{d}x_2)}\cdots I(\mathrm{d}x_p)\\
\vdots\\
I(\mathrm{d}x_1)\cdots I(\mathrm{d}x_{p-1})\widehat{I(\mathrm{d}x_p)}
\end{array}\right)=\underbrace{\left(\begin{array}{cccc}
\mathsf{M}_{11} & \mathsf{M}_{12} & \cdots &\mathsf{M}_{1p}\\
\mathsf{M}_{21} & \mathsf{M}_{22} & \cdots & \mathsf{M}_{2p}\\
\vdots & \vdots & \ddots & \vdots \\
\mathsf{M}_{p1} & \mathsf{M}_{p2} & \cdots  &\mathsf{M}_{pp}
\end{array}\right)}_{\mathsf{M}_{\mathcal{J}}}\underbrace{\left(\begin{array}{c}
\widehat{\mathrm{d}x_1}\mathrm{d}x_2\cdots\mathrm{d}x_p\\
\mathrm{d}x_1\widehat{\mathrm{d}x_2}\cdots\mathrm{d}x_p\\
\vdots\\
\mathrm{d}x_1\cdots\mathrm{d}x_{p-1}\widehat{\mathrm{d}x_p}
\end{array}\right)}_{\mathrm{d}^{p-1}\widehat{\mathbf{x}}},
\end{equation}
where $\mathsf{M}_{\mathcal{J}}$ is the matrix of first minors of the Jacobian matrix $\mathcal{J}$. Thus, (\ref{almostthere}) reduces to
\begin{equation}\label{firstpart}
\sum_{k=1}^p \mathrm{d}\left(\mathcal{A}d(\rho)^{-1}\left(\sum_\alpha(-1)^{k-1} \mathscr{Q}^\alpha(J,I)\mathcal{C}^\alpha_k \right)\mathsf{M}_\mathcal{J}\mathrm{d}^{p-1}\widehat{\mathbf{x}}\right), 
\end{equation}
and we have thus found the invariantized version of $\mathsf{Div}(\mathrm{A})$. We must now find the invariantized version of the term $\mathsf{Div}(\bar{L}\Xi_{i})$ in the infinitesimal criterion of invariance, for $i=1,...,r$. We know from Theorem \ref{adjointactioncons} that
$$\begin{array}{l}
\left(\begin{array}{c}
\displaystyle{\sum_{k=1}^p (-1)^{k-1}C^1_k(\widetilde{\mathbf{x}},\widetilde{\mathbf{u}},\widetilde{\mathbf{u_x}},\Xi_1(\widetilde{\mathbf{x}},\widetilde{\mathbf{u}}),\Phi_1(\widetilde{\mathbf{x}},\widetilde{\mathbf{u}}))\mathrm{d}\widetilde{x_1}...\,\widehat{\mathrm{d}\widetilde{x_k}}...\,\mathrm{d}\widetilde{x_p}}\\
\vdots\\
\displaystyle{\sum_{k=1}^p (-1)^{k-1}C^r_k(\widetilde{\mathbf{x}},\widetilde{\mathbf{u}},\widetilde{\mathbf{u_x}},\Xi_1(\widetilde{\mathbf{x}},\widetilde{\mathbf{u}}),\Phi_1(\widetilde{\mathbf{x}},\widetilde{\mathbf{u}}))\mathrm{d}\widetilde{x_1}...\,\widehat{\mathrm{d}\widetilde{x_k}}...\,\mathrm{d}\widetilde{x_p}}
\end{array}\right)\\[60pt]
\quad=\mathcal{A}d(g)\left(\begin{array}{c}
\displaystyle{\sum_{k=1}^p (-1)^{k-1}C^1_k(\mathbf{x},\mathbf{u},\mathbf{u_x},\Xi_1(\mathbf{x},\mathbf{u}),\Phi_1(\mathbf{x},\mathbf{u}))\mathrm{d}x_1...\,\widehat{\mathrm{d}x_k}...\,\mathrm{d}x_p}\\
\vdots\\
\displaystyle{\sum_{k=1}^p (-1)^{k-1}C^r_k(\mathbf{x},\mathbf{u},\mathbf{u_x},\Xi_1(\mathbf{x},\mathbf{u}),\Phi_1(\mathbf{x},\mathbf{u}))\mathrm{d}x_1...\,\widehat{\mathrm{d}x_k}...\,\mathrm{d}x_p}
\end{array}\right).
\end{array}$$
Thus,
$$\begin{array}{l}
\displaystyle{\sum_{k=1}^p(-1)^{k-1}\bar{L}(\widetilde{\mathbf{x}},\widetilde{\mathbf{u}}\widetilde{,u^\alpha_\mathrm{K}})(\Xi(\widetilde{\mathbf{x}},\widetilde{\mathbf{u}}))_k\mathrm{d}\widetilde{x_1}...\,\widehat{\mathrm{d}\widetilde{x_k}}...\,\mathrm{d}\widetilde{x_p}}\\
\quad\displaystyle{= \mathcal{A}d(g)\sum_{k=1}^p(-1)^{k-1}\bar{L}(\mathbf{x},\mathbf{u},u^\alpha_\mathrm{K})(\Xi(\mathbf{x},\mathbf{u}))_k\mathrm{d}x_1...\,\widehat{\mathrm{d}x_k}...\,\mathrm{d}x_p,}
\end{array}$$
where $(\Xi(\mathbf{x},\mathbf{u}))_k$ is the $k^\mathrm{th}$ column of $\Xi(\mathbf{x},\mathbf{u})$. Evaluating this at the frame and rearranging produces the boundary term, $\mathsf{Div}(\bar{L}(\Xi)_k)$,
\begin{equation}\label{secondpart}
\mathrm{d}\left(\displaystyle{\mathcal{A}d(\rho)^{-1}
\sum_{k=1}^p(-1)^{k-1}L[\kappa](\Xi(J,I))_k I(\mathrm{d}x_1)...\,\widehat{I(\mathrm{d}x_k)}...\,I(\mathrm{d}x_p)}\right).
\end{equation}
Thus, adding the boundary terms (\ref{firstpart}) and (\ref{secondpart}) yields 
$$\mathrm{d} \left(\mathcal{A}d(\rho)^{-1}\left(\boldsymbol{\upsilon}_1,\cdots,\boldsymbol{\upsilon}_p\right)\mathsf{M}_{\mathcal{J}}\,\mathrm{d}^{p-1}\widehat{\mathbf{x}}\,\right)=0,$$
as required. $\hfill \Box$
\end{proof}

In terms of calculating the conservation laws in the form
$$\mathrm{d}\left(\mathcal{A}d(\rho)^{-1}(\boldsymbol{\upsilon}_1,...,\boldsymbol{\upsilon}_p)\mathsf{M}_\mathcal{J}\mathrm{d}^{p-1}\widehat{\mathbf{x}}\right)=0,$$
the vectors of invariants can be obtained by either
\begin{enumerate}
\item invariantization of the components of the law in the original coordinates, or
\item using the formula (\ref{vectorofinvs}).
\end{enumerate}
As there exists software which calculates the conservation laws (Maple package JetCalculus), it will usually be easier to invariantize the conservation laws to obtain the vectors of invariants, rather than perform the invariantized integration by parts.

\begin{example}\label{illustration}
Here we illustrate how the different components of the conservation laws in Example \ref{conservationlawsmonge} are obtained. {We have already obtained the Adjoint representation $\mathcal{A}d(g)$ for SL(2) in Example \ref{SL2adjointrep}. Inverting $\mathcal{A}d(g)$ in (\ref{Adjointrepresentation}) and evaluating it  at the frame (\ref{framesl2}) yields $\mathcal{A}d(\rho)^{-1}$.}

Theorem \ref{conslawsnoninvindvar} tells us that to obtain the vectors of invariants,  we need to compute the invariantized matrix of characteristics, $\mathscr{Q}^u(J,I)$, and the vectors $\mathcal{C}^u_i$. The latter have already been calculated in Example \ref{exampleeleqnbts}; the elements of $\mathcal{C}^u_i$ correspond to the coefficients of the $I^\alpha _{\mathrm{J}\tau}$ in (\ref{vectorsC-alpha-i}). The invariantized matrix of characteristics is
$$\mathscr{Q}^u(J,I)=\bordermatrix{ & Q_u & D_x( Q_u) & D_y( Q_u)\cr
a & -I^u_1 & -I^u_1-I^u_{11} & -I^u_{12} \cr
b & 0 & 0 & -I^u_1\cr
c & 0 & -I^u_{12} & -I^u_{22}\cr}$$
and the $(\Xi(J,I))_i$, for $i=1,2$, are
$$(\Xi(J,I))_1=\bordermatrix{ & \xi^x\cr
a & 1\cr b & 0\cr c & 0\cr},\qquad (\Xi(J,I))_2=\bordermatrix{ & \xi^y\cr
a & 0\cr b & 0\cr c & 1\cr},$$
Thus, the vectors of invariants are
$$\begin{array}{c}
\boldsymbol{\upsilon}_1= \left(\begin{array}{c}
I^u_1I^u_{22}(I^u_1-2I^u)-I^uI^u_1I^u_{122}+I^u(I^u_{11}I^u_{22}-(I^u_{12})^2)\\
0\\
-I^uI^u_{12}I^u_{22}\end{array}\right),\\[40pt]
\boldsymbol{\upsilon}_2= \left(\begin{array}{c}
-I^uI^u_1(2I^u_{12}+I^u_{112})\\
I^uI^u_1I^u_{11}\\
-I^u(I^u_{12})^2\end{array}\right).
\end{array}$$
Finally, the Jacobian matrix $\mathcal{J}$ is
$$\left(\begin{array}{cc}
\left.\dfrac{\partial \widetilde{x}}{\partial x}\right|_{g=\rho(\zede)} & \left.\dfrac{\partial \widetilde{x}}{\partial y}\right|_{g=\rho(\zede)}\\[10pt]
\left.\dfrac{\partial \widetilde{y}}{\partial x}\right|_{g=\rho(\zede)} & \left.\dfrac{\partial \widetilde{y}}{\partial y}\right|_{g=\rho(\zede)}\end{array}\right)=\left(\begin{array}{cc}
\dfrac{u_x}{xu_x+yu_y} & \dfrac{u_y}{xu_x+yu_y}\\
-y & x\end{array}\right),$$
and its matrix of first minors ,$\mathsf{M}_{\mathcal{J}}$, is
$$\left(\begin{array}{cc}
x & -y\\
\dfrac{u_y}{xu_x+yu_y} & \dfrac{u_x}{xu_x+yu_y}\end{array}\right).$$

Although the vectors of invariants obtained here are not the same as those obtained in Example \ref{conservationlawsmonge} (these were obtained by invariantizing the laws), the resulting conservation laws are equivalent, i.e. the conservation laws differ by a trivial conservation laws. These are of two types: the
 first kind, where the trivial conservation law vanishes on all solutions of the given system, or,
 second kind, where it holds for any smooth function $\mathbf{u}=f(\mathbf{x})$.

To conclude this example, we summarise the information made available by employing the invariant calculus for this group action. For the frame with normalisation equations $\widetilde{x}=1$, $\widetilde{y}=0$ and $\widetilde{u_y}=0$, the differential algebra of invariants is generated by $u$ and $I(u_{yy})$. In addition to the Euler-Lagrange equation, which is now seen to be one equation for the two generators, there is also the syzygy, Equation (\ref{MainExSyzA}), providing a second equation connecting the generating invariants. In this case we can calculate the frame which is given in Equation (\ref{framesl2}). The invariant differentiation operators are given in Equations (\ref{MainEGInvDiffOps}) and (\ref{MainEGInvDiffOps1}), and setting the frame into the standard $2\times 2$ matrix form we have
\begin{equation}\label{MainExRhoEqs}
\mathcal{D}_x\rho \rho^{-1}=\left(\begin{array}{cc}1&\frac{\mathcal{D}_y\mathcal{D}_x(u)}{\mathcal{D}_x(u)}\\0&1\end{array}\right),\qquad 
\mathcal{D}_y\rho \rho^{-1}=\left(\begin{array}{cc}0&\frac{I(u_{yy})}{\mathcal{D}_x(u)}\\-1&0\end{array}\right).
\end{equation}
The differential compatibility of these equations also yields the syzygy between the generating invariants. Finally, we have the conservation laws, which when differentiated yield the Euler-Lagrange equation. Finally, we note that the frame, its Adjoint representation, the differential operators, the syzygies and the equations connecting the derivatives of the frame with the invariants, all remain unchanged as the Lagrangian is varied, so that these are a ``one time" calculation once the equations for the frame are chosen.
\end{example}
 
\section{Two variational problems with area and volume preserving symmetries}\label{twoexamples}
In this section, we present two exampless which illustrate how to obtain the conservation laws in this new format. The first example regards the conservation laws for the shallow water equations, due to the importance that conservation of potential vorticity plays in meteorology. In the second application we look at conservation laws arising from a linear SL(3) action on the base space, as it exemplifies the basic volume preserving action on a three-dimensional base space.

\subsection{Conservation laws for the shallow water equations}
The conservation laws for the shallow water equations are well-known \cite{BilaClarksonMansfield}; we are particularly interested in the conservation laws arising from the linear SL(2) action on the particle labels.

To ease the exposition, some notation is introduced. In the two-dimensional shallow water theory \cite{RubstovRoulstone}, a particle is represented by the Cartesian coordinates
\begin{equation}\label{cartcoord}
x=x(a,b,t),\qquad y=y(a,b,t),
\end{equation}
where $(a,b)\in \mathbb{R}^2$ are the particle labels and $t\in \mathbb{R}^+$ is time. At the reference time, $t=0$, 
$$x(a,b,0)=a,\qquad y(a,b,0)=b.$$
Usually we regard liquids, such as water, to be incompressible; the incompressibility hypothesis requires that
$$\dfrac{h(a,b,0)}{h(a,b,t)}=\dfrac{\partial(x,y)}{\partial(a,b)},$$
where $h$ is the fluid depth, and the Jacobian on the right is the one corresponding to the map (\ref{cartcoord}). In this paper we assume that $h(a,b,0)=1$, so the incompressibility hypothesis becomes
\begin{equation}
h(a,b,t)=\dfrac{1}{x_ay_b-x_by_a}.
\end{equation}

As shown by Salmon \cite{Salmon}, the following first order Lagrangian
\begin{equation}\label{Lagrangianswe}
\bar{L}\,\mathrm{d}a\mathrm{d}b\mathrm{d}t=\left((u-\bar{R})\dot{x}+(v+\bar{P})\dot{y}-\dfrac{1}{2}(u^2+v^2+gh)\right)\,\mathrm{d}a\mathrm{d}b\mathrm{d}t,
\end{equation}
where $g$ is a nonzero constant (corresponding to the combined effect of acceleration of gravity and a centrifugal component from the Earth's rotation), $\bar{P}=\bar{P}(x,y)$ and $\bar{R}=\bar{R}(x,y)$ satisfy 
$$\bar{P}_x+\bar{R}_y=f,\qquad \textrm{with the Coriolis parameter, }f=\textrm{constant},$$
has the shallow water equations
\begin{align}
\dot{x}&=u,\\[-10pt]
\dot{y}&=v,\\[-10pt]
\dot{u}&+gh(y_bh_a-y_ah_b)-fv=0,\\[-10pt]
\dot{v}&+gh(x_ah_b-x_bh_a)+fu=0,
\end{align}
 as the associated Euler-Lagrange equations.
 
 To simplify we will consider $\bar{P}$ and $\bar{R}$ to be linear functions of $x$ and $y$, i.e. 
 $$\bar{P}=c_1x+c_2y+c_3\quad \mathrm{and}\quad \bar{R}=c_4x+c_5y+c_6.$$
 The following vector field
  $$-S_b(a,b)\partial_a+S_a(a,b)\partial_b,\qquad S_b=-\xi, S_a=\eta,$$
 where $\xi$ and $\eta$ are the infinitesimals of the group action on the base space, generates the particle relabelling symmetry group \cite{BilaClarksonMansfield}. The generators of the linear SL(2) action are of this type; the action is
 $$\left(\begin{array}{c}
 \widetilde{a}\\ \widetilde{b}\end{array}\right)=\left(\begin{array}{cc} \alpha& \beta\\ \gamma& \delta\end{array}\right)\left(\begin{array}{c}a\\b\end{array}\right),\qquad \widetilde{t}=t, \qquad \alpha\delta-\beta\gamma=1.$$
We now find  the associated conservation laws.

We start by calculating the moving frame using as normalization equations
$$\widetilde{a}=0, \qquad \widetilde{b}=1,\qquad \widetilde{x_a}=0,$$
which yields
\begin{equation}\label{frameofswe}
\alpha=b,\qquad \beta=-a,\qquad \gamma=\dfrac{x_a}{ax_a+bx_b},
\end{equation}
as the moving frame in parametric form.

We already have the adjoint representation for SL(2) (see (\ref{Adjointrepresentation})); so evaluating it at the frame (\ref{frameofswe}) and inverting it gives $\mathcal{A}d(\rho)^{-1}$ (see first matrix of (\ref{consforswe})). Next we need to compute the vectors of invariants. For this, we introduce a dummy variable $\tau$ and set $x=x(a,b,t,\tau)$, $y=y(a,b,t,\tau)$, $u=u(a,b,t,\tau)$, and $v=v(a,b,t,\tau)$. Proceeding as in Section (\ref{structureNCL}), we rewrite the Lagrangian (\ref{Lagrangianswe}) in terms of the invariants; then differentiating and integrating by parts yields the invariantized shallow water equations 
\small{\begin{align}\nonumber
&f I^y_3-I^u_3+\frac{gI^y_2}{(I^x_2)^3(I^y_1)^3}(I^y_{11}I^x_2-I^x_{11}I^y_2+I^x_{12}I^y_1)+\frac{g}{(I^x_2)^3(I^y_1)^2}(I^x_{12}I^y_2-I^y_{12}I^x_2-I^x_{22}I^y_1)=0,\\\nonumber
&-f I^x_3-I^v_3-\frac{g}{(I^x_2)^2(I^y_1)^3}(I^y_{11}I^x_2-I^x_{11}I^y_2+I^x_{ 12}I^y_1)=0,\\\nonumber
&I^x_3-I^u=0,\\\nonumber
&I^y_3-I^v=0,
\end{align}}\normalsize
as expected, and the boundary terms
\small{\begin{align}\nonumber
&\mathcal{D}_a\left(\left(\frac{gI^y_2I^x_4}{2(I^x_2)^2(I^y_1)^2}-\frac{gI^y_4}{2I^x_2(I^y_1)^2}\right)I(\mathrm{d}a)I(\mathrm{d}b)I(\mathrm{d}t)\right)+\mathcal{D}_b\left(\left(-\frac{gI^x_4}{2(I^x_2)^2I^y_1}\right)I(\mathrm{d}a)I(\mathrm{d}b)I(\mathrm{d}t)\right)\\\nonumber
&\quad+\mathcal{D}_t\left(\left( (I^u-R)I^x_4+(I^v+P)I^y_4\right)I(\mathrm{d}a)I(\mathrm{d}b)I(\mathrm{d})t\right)=0,
\end{align}}\normalsize
where $P$ and $R$ are the invariantized versions of $\bar{P}$ and $\bar{R}$, respectively.

Thus, the vectors of invariants are
$$\boldsymbol{\upsilon}_1(J,I)=\underbrace{\begin{pmatrix}I^x_2\\0\\0\end{pmatrix}}_{\mathscr{Q}^x}\frac{gI^y_2}{2(I^x_2)^2(I^y_1)^2}-\underbrace{\begin{pmatrix}I^y_2\\-I^y_1\\0\end{pmatrix}}_{\mathscr{Q}^y}\frac{g}{2I^x_2(I^y_1)^2}+L\underbrace{\begin{pmatrix}0\\1\\0 \end{pmatrix}}_{(\Xi)_1}=\begin{pmatrix}0\\L+\dfrac{g}{2I^x_2I^y_1}\\0\end{pmatrix},$$
$$\boldsymbol{\upsilon}_2(J,I)=\underbrace{\begin{pmatrix}I^x_2\\0\\0\end{pmatrix}}_{\mathscr{Q}^x}\frac{g}{2(I^x_2)^2I^y_1}-L\underbrace{\begin{pmatrix}-1\\0\\0\end{pmatrix}}_{(\Xi)_2}=\begin{pmatrix}L+\dfrac{g}{2I^x_2I^y_1}\\0\\0\end{pmatrix},$$
$$\boldsymbol{\upsilon}_3(J,I)=\underbrace{\begin{pmatrix}I^x_2\\0\\0\end{pmatrix}}_{\mathscr{Q}^x}(I^u-R)+\underbrace{\begin{pmatrix}I^y_2\\-I^y_1\\0\end{pmatrix}}_{\mathscr{Q}^y}(I^v+P)+L\underbrace{\begin{pmatrix}0\\0\\0 \end{pmatrix}}_{(\Xi)_3}=\begin{pmatrix}I^x_2(I^u-R)+I^y_2(I^v+P)\\-I^y_1(I^v+P)\\0\end{pmatrix}.$$

The matrix of first minors of the Jacobian matrix $\frac{\partial(\widetilde{a},\widetilde{b},\widetilde{t})}{\partial(a,b,t)}$ evaluated at the frame (\ref{frameofswe}) is
$$\mathsf{M}_{\mathcal{J}}=\begin{pmatrix}\dfrac{x_b}{ax_a+bx_b}& \dfrac{x_a}{ax_a+bx_b} & 0\\
-a & b & 0\\  0 & 0 & 1\end{pmatrix}.$$

Thus, the conservation laws are 
\begin{equation}\label{consforswe}\begin{array}{c}
\mathrm{d}\left(\begin{pmatrix}\dfrac{bx_b-ax_a}{ax_a+bx_b} & 2ab & \dfrac{2x_ax_b}{(ax_a+bx_b)^2}\\-\dfrac{bx_a}{ax_a+bx_b} & b^2 & -\dfrac{x_a^2}{(ax_a+bx_b)^2}\\
-\dfrac{ax_b}{ax_a+bx_b} & -a^2 & \dfrac{x_b^2}{(ax_a+bx_b)^2}\end{pmatrix}\right.\\
\qquad\qquad\qquad\qquad\times\begin{pmatrix} 0 & L+\dfrac{g}{2I^x_2I^y_1} & I^x_2(I^u-R)+I^y_2(I^v+P)\\
L+\dfrac{g}{2I^x_2I^y_1}  & 0& -I^y_1(I^v+P)\\
0 & 0 & 0\end{pmatrix}\\
\qquad\qquad\qquad\qquad\left.\times\begin{pmatrix}\dfrac{x_b}{ax_a+bx_b}& \dfrac{x_a}{ax_a+bx_b} & 0\\
-a & b & 0\\  0 & 0 & 1\end{pmatrix}\begin{pmatrix}\mathrm{d}b\mathrm{d}t\\ \mathrm{d}a\mathrm{d}t\\
\mathrm{d}a\mathrm{d}b\end{pmatrix}\right)=0.\end{array}
\end{equation}
Note that $L=\bar{L}(I)$.

In \cite{BridgesHydonReich} Bridges et al. proved that conservation of potential vorticity is a differential consequence of  some of the components of the 1-form quasi-conservation law, which relies on writing the shallow water equations as a multisymplectic system. Below we show that  conservation of potential vorticity is a differential consequence of the system of conservation laws (\ref{consforswe}).

Multiplying (\ref{consforswe}) through, we obtain
\begin{align}\label{consA}
&\mathrm{d}\left(\left(a\mathrm{F}_1\right)\mathrm{d}b\mathrm{d}t+\left(b\mathrm{F}_1\right)\mathrm{d}a\mathrm{d}t+\left(\dfrac{bx_b-ax_a}{ax_a+bx_b}\,\mathrm{F}_2-2ab\;\mathrm{F}_3\right)\mathrm{d}a\mathrm{d}b\right)=0,\\\label{consB}
&\mathrm{d}\left(\left(b\mathrm{F}_1\right)\mathrm{d}b\mathrm{d}t+\left(-\dfrac{bx_a}{ax_a+bx_b}\,\mathrm{F}_2-b^2\;\mathrm{F}_3\right)\mathrm{d}a\mathrm{d}b\right)=0,\\\label{consC}
&\mathrm{d}\left(-\left(a\mathrm{F}_1\right)\mathrm{d}a\mathrm{d}t+\left(-\dfrac{ax_b}{ax_a+bx_b}\,\mathrm{F}_2+a^2\;\mathrm{F}_3\right)\mathrm{d}a\mathrm{d}b\right)=0,
\end{align}
where $\mathrm{F}_1=L+g/(2I^x_2I^y_1)$, $\mathrm{F}_2=I^x_2(I^u-R)+I^y_2(I^v+P)$, and $\mathrm{F}_3=I^y_1(I^v+P)$. Performing the following operations, $D_a\left(b\cdot (\ref{consC})\right)-D_b\left(a\cdot (\ref{consB})\right) + (\ref{consA})$, on the above equations we obtain
\begin{align*}
&\left(D_a\left(D_b(ab\mathrm{F}_1)-a\mathrm{F}_1\right)+D_a\left(b D_t\left(\dfrac{-ax_b}{ax_a+bx_b}\mathrm{F}_2+a^2\mathrm{F}_3\right)\right)\right)\mathrm{d}a\mathrm{d}b\mathrm{d}t\\
&-\left(D_b(D_a(ab\mathrm{F}_1)-b\mathrm{F}_1)+D_b\left(aD_t\left(-\dfrac{bx_a}{ax_a+bx_b}\mathrm{F}_2-b^2\mathrm{F}_3\right)\right)\right)\mathrm{d}a\mathrm{d}b\mathrm{d}t\\
&+\left(D_a(a\mathrm{F}_1)-D_b(b\mathrm{F}_1)+D_t\left(\dfrac{bx_b-ax_a}{ax_a+bx_b}\mathrm{F}_2-2ab\mathrm{F}_3\right)\right)\mathrm{d}a\mathrm{d}b\mathrm{d}t\\
&=D_t\left(D_a\left(\dfrac{-abx_b}{ax_a+bx_b}\mathrm{F}_2+a^2b\mathrm{F}_3\right)+D_b\left(\dfrac{abx_a}{ax_a+bx_b}\mathrm{F}_2+ab^2\mathrm{F}_3\right)\right)\mathrm{d}a\mathrm{d}b\mathrm{d}t\\
&=D_t\left(ab\dfrac{I^x_{12}}{I^x_2}\mathrm{F}_2+2ab\mathrm{F}_3-ab\mathcal{D}_a\mathrm{F}_2+ab\mathcal{D}_b\mathrm{F}_3\right)\mathrm{d}a\mathrm{d}b\mathrm{d}t\\
&=-ab D_t\left(I^u_1I^x_2+I^y_2I^v_1-I^y_1I^v_2-I^x_2I^y_1f \right)\mathrm{d}a\mathrm{d}b\mathrm{d}t\\
&=-ab D_t(\Omega)=0,
\end{align*}
where $\Omega=1/h(\partial \dot{y}/\partial x-\partial\dot{x}/\partial y+f)$ represents the potential vorticity. Note that we have used the product rule and the definitions of the invariantized differential operators $\mathcal{D}_a$ and $\mathcal{D}_b$. Thus, conservation of potential vorticity is a differential consequence of Noether's conservation laws for the linear SL(2) action. More to the point, it does not require the full pseudogroup. This was also observed by Hydon, \cite{Hydon}, who found the conservation of potential vorticity as a differential consequence of the conservation of the linear momenta.

\subsection{Invariant variational problems under the SL(3) action}
Consider the linear SL(3) action on the base space $(x,y,z)$,
\begin{equation}\label{linearsl3}
\left(\begin{array}{c}
\widetilde{x}\\\widetilde{y}\\\widetilde{z}\end{array}\right)=\underbrace{\left(\begin{array}{ccc}a_{11} & a_{12} & a_{13}\\
a_{21} & a_{22} & a_{23}\\
a_{31} & a_{32} & a_{33}\end{array}\right)}_A\left(\begin{array}{c}x\\y\\z\end{array}\right), \qquad \det{A}=1,
\end{equation}
which leaves the dependent variables, $(u,v,w),$ invariant. 

Let $g\in \mathrm{SL(3)}$ act on the Jacobian $B=\frac{\partial(u,v,w)}{\partial(x,y,z)}$ and define the cross section by
\begin{equation}
g\cdot \frac{\partial(u,v,w)}{\partial(x,y,z)}=\left(\begin{array}{ccc} 1 & 0 & 0\\
0 & 1 & 0\\
0 & 0 & I^w_3\end{array}\right),
\end{equation}
where $I^w_3=(g\cdot w_z)|_{\mathrm{frame}}$. Thus, the moving frame in parametric form is
\begin{equation}\label{sl3frame}
(a_{11}, a_{12}, a_{13}, a_{21}, a_{22}, a_{23}, a_{31}, a_{32})=\left(u_x, u_y, u_z, v_x, v_y, v_z, \dfrac{w_x}{|B|},  \dfrac{w_y}{|B|}\right).
\end{equation}

Consider an invariant variational problem, written in terms of the invariants of the group action (\ref{linearsl3}), such as
\begin{equation}\label{sl3functional}
\iiint L(I^w,\mathcal{D}_z I^w) I(\mathrm{d}x)I(\mathrm{d}y)I(\mathrm{d}z).
\end{equation}
To calculate the invariantized Euler-Lagrange equations and its associated conservation laws, we introduce a dummy variable $\tau$ and set $u=u(x,y,z,\tau)$, $v=v(x,y,z,\tau)$, and $w=w(x,y,z,\tau)$.
Differentiating the functional (\ref{sl3functional}) in terms of $\tau$ and integrating by parts, we obtain
\begin{align}\nonumber
\mathcal{D}_\tau&\iiint L(I^w,\mathcal{D}_zI^w)I(\mathrm{d}x)I(\mathrm{d}y)I(\mathrm{d}z)\\\nonumber
=&\iiint \Big[-\mathcal{D}_x\left(\frac{\partial L}{\partial \mathcal{D}_zI^w}\right)I^w_3I^u_4-\mathcal{D}_y\left(\frac{\partial L}{\partial \mathcal{D}_zI^w}\right)I^w_3I^v_4+\left(\frac{\partial L}{\partial I^w}-\left(\frac{\partial L}{\partial\mathcal{D}_zI^w}\right)\right)I^w_4\Big]I(\mathrm{d}\mathbf{x})\\\label{ELeqnsBTs}
+&\iiint \Big[\mathcal{D}_x\left(\frac{\partial L}{\partial \mathcal{D}_zI^w}I^w_3I^u_4I(\mathrm{d}\mathbf{x})\right)+\mathcal{D}_y\left(\frac{\partial L}{\partial \mathcal{D}_zI^w}I^w_3I^v_4I(\mathrm{d}\mathbf{x})\right)+\mathcal{D}_z\left(\frac{\partial L}{\partial \mathcal{D}_zI^w}I^w_4I(\mathrm{d}\mathbf{x})\right)\Big],
\end{align}
where we have used the equality $\mathcal{D}_zI^w=I^w_3$, the commutator 
$$[\mathcal{D}_\tau,\mathcal{D}_z]=-\mathcal{D}_zI^u_4\mathcal{D}_x-\mathcal{D}_zI^v_4\mathcal{D}_y+(\mathcal{D}_xI^u_4+\mathcal{D}_yI^v_4)\mathcal{D}_z,$$
and the Lie derivatives of the invariant one-forms presented in the following table.

\begin{table}[h]
\begin{center}
\begin{tabular}{c|c|c|c|c}
& $I(\mathrm{d}x)$ & $I(\mathrm{d}y)$ & $I(\mathrm{d}z)$ & $I(\mathrm{d}\tau)$\\
\hline
$\mathcal{D}_x$ & \small{$-I^u_{14} I(\mathrm{d}\tau)$} \normalsize& \small{$-I^v_{14} I(\mathrm{d}\tau)$}\normalsize  & \small{$-\left(I^u_{11}+I^v_{12}+\dfrac{I^w_{13}}{I^w_3}\right) I(\mathrm{d}z)$}\normalsize & \small{0}\normalsize\\
& & & \small{$-\dfrac{I^w_{14}}{I^w_3} I(\mathrm{d}\tau)$}\normalsize & \\
$\mathcal{D}_y$ & \small{$-I^u_{24} I(\mathrm{d}\tau)$}\normalsize & \small{$-I^v_{24} I(\mathrm{d}\tau)$}\normalsize  & \small{$-\left(I^u_{12}+I^v_{22}+\dfrac{I^w_{23}}{I^w_3}\right) I(\mathrm{d}z)$}\normalsize & \small{0}\normalsize\\
& & & \small{$-\dfrac{I^w_{24}}{I^w_3} I(\mathrm{d}\tau)$}\normalsize & \\
$\mathcal{D}_z$ & \small{$-I^u_{34} I(\mathrm{d}\tau)$}\normalsize & \small{$-I^v_{34} I(\mathrm{d}\tau)$}\normalsize  & \small{$\left(I^u_{11}+I^v_{12}+\dfrac{I^w_{13}}{I^w_3}\right)I(\mathrm{d}x)$}\normalsize & \small{0}\normalsize\\
& & & \small{$+\left(I^u_{12}+I^v_{22}+\dfrac{I^w_{23}}{I^w_3}\right) I(\mathrm{d}y)$}\normalsize & \\
& & & \small{$+(I^u_{14}+I^v_{24})I(\mathrm{d}\tau)$}\normalsize & \\
$\mathcal{D}_\tau$ & \small{$I^u_{14}I(\mathrm{d}x)+I^u_{24}I(\mathrm{d}y)$}\normalsize & \small{$I^v_{14}I(\mathrm{d}x)+I^v_{24}I(\mathrm{d}y)$}\normalsize & \small{$\dfrac{I^w_{14}}{I^w_3}I(\mathrm{d}x)+\dfrac{I^w_{24}}{I^w_3}I(\mathrm{d}y)$}\normalsize & \small{0}\normalsize\\
& \small{$I^u_{34}I(\mathrm{d}z)$}\normalsize & \small{$I^v_{34}I(\mathrm{d}z)$}\normalsize & \small{$-(I^u_{14}+I^v_{24})I(\mathrm{d}z)$}\normalsize & 
\end{tabular}
\textbf{\caption{Lie derivatives of the invariant one-forms.}}
\end{center}
\end{table}

Notice that the coefficients of $I^u_4$, $I^v_4$, and $I^w_4$ in (\ref{ELeqnsBTs}), which are not in the boundary terms, correspond to the invariantized Euler-Lagrange equations with respect to $u$, $v$, and $w$, respectively.

Proceeding as in Section \ref{structureNCL}, we let $g\in \mathrm{SL(3)}$ act  linearly on its generating vector fields
$$x\partial_x-z\partial_z,\quad y\partial_x,\quad z\partial_x,\quad x\partial_y,\quad y\partial_y-z\partial_z,\quad z\partial_y,\quad x\partial_z,\quad y\partial_z.$$
This yields the adjoint representation, $\mathcal{A}d(g)$, for SL(3)
\begin{equation}
\left(\begin{array}{cccccccc}
M_{11}A-M_{31}[R_3,\mathbf{0},\mathbf{0}]^T & -M_{12}A+M_{32}[R_3,\mathbf{0},\mathbf{0}]^T & M_{13}[C_1,C_2]-M_{33}\bigl(\begin{smallmatrix}
a_{31}&a_{32}\\ 0&0\\0&0
\end{smallmatrix} \bigr)\\
-M_{21}A-M_{31}[\mathbf{0},R_3,\mathbf{0}]^T & M_{22}A+M_{32}[\mathbf{0},R_3,\mathbf{0}]^T & -M_{23}[C_1,C_2]-M_{33}\bigl(\begin{smallmatrix}
0&0\\a_{31}&a_{32}\\0&0
\end{smallmatrix} \bigr)\\\
M_{31}[R_1,R_2]^T & -M_{32}[R_1,R_2]^T & M_{33}\bigl(\begin{smallmatrix}
a_{11}&a_{12}\\ a_{21}&a_{22}
\end{smallmatrix} \bigr)
\end{array}\right),
\end{equation}
where the column vectors $R_i$, for $i=1,2,3$, and $C_j$, for $j=1,2$, represent, respectively, the rows and columns of matrix $A$ defined in (\ref{linearsl3}), the $M_{mn}$, for $m,n=1,2,3$, represent the first minors of $A$, and the $a_{mn}$ are elements of the matrix $A$. Evaluating $\mathcal{A}d(g)^{-1}$ at the frame (\ref{sl3frame}) yields $\mathcal{A}d(\rho)^{-1}$.

The vectors of invariants, $\boldsymbol{\upsilon}_i=(-1)^{i-1}\left(\sum_\alpha\mathscr{Q}^\alpha(J,I)\mathcal{C}^\alpha_i+L(\Xi(J,I))_i\right)$, are
\begin{equation*}
\boldsymbol{\upsilon}_1(J,I)=\kern-2pt\left(\begin{array}{c}J^x\left(L-I^w_3\dfrac{\partial L}{\partial\mathcal{D}_xI^w}\right)\\
J^y\left(L-I^w_3\dfrac{\partial L}{\partial\mathcal{D}_xI^w}\right)\\
J^z\left(L-I^w_3\dfrac{\partial L}{\partial\mathcal{D}_xI^w}\right)\\
0\\0\\0\\0\\0\end{array}\right),\qquad
\boldsymbol{\upsilon}_2(J,I)=\kern-2pt\left(\begin{array}{c}0\\0\\0\\J^x\left(-L+I^w_3\dfrac{\partial L}{\partial\mathcal{D}_xI^w}\right)\\
J^y\left(-L+I^w_3\dfrac{\partial L}{\partial\mathcal{D}_xI^w}\right)\\
J^z\left(-L+I^w_3\dfrac{\partial L}{\partial\mathcal{D}_xI^w}\right)\\
0\\0\end{array}\right),\end{equation*}
\begin{equation*}
\boldsymbol{\upsilon}_3(J,I)=\kern-2pt\left(\begin{array}{c}J^z\left(-L+I^w_3\dfrac{\partial L}{\partial\mathcal{D}_xI^w}\right)\\0\\0\\0\\J^z\left(-L+I^w_3\dfrac{\partial L}{\partial\mathcal{D}_xI^w}\right)\\0\\
J^x\left(L-I^w_3\dfrac{\partial L}{\partial\mathcal{D}_xI^w}\right)\\
J^y\left(L-I^w_3\dfrac{\partial L}{\partial\mathcal{D}_xI^w}\right)\end{array}\right),
\end{equation*}
where we have used 
$$\mathscr{Q}^u(J,I)=\left(\begin{array}{c}
-J^x\\ -J^y\\ -J^z\\ 0\\ 0\\ 0\\ 0\\ 0\end{array}\right),\quad \mathscr{Q}^v(J,I)=\left(\begin{array}{c}
0\\ 0\\ 0\\-J^x\\ -J^y\\ -J^z\\ 0\\ 0\\\end{array}\right),\quad \mathscr{Q}^w(J,I)=\left(\begin{array}{c}
J^zI^w_3\\ 0\\ 0\\ 0\\ J^zI^w_3\\ 0\\ -J^xI^w_3\\ -J^yI^w_3\end{array}\right),$$
$$(\Xi(J,I))_1=\left(\begin{array}{c}J^x\\J^y\\J^z\\0\\0\\0\\0\\0\end{array}\right),\quad (\Xi(J,I))_2=\left(\begin{array}{c}0\\0\\0\\J^x\\J^y\\J^z\\0\\0\end{array}\right),\quad(\Xi(J,I))_3=\left(\begin{array}{c}-J^z\\0\\0\\0\\-J^z\\0\\J^x\\J^y\end{array}\right).$$

Finally, we calculate the last component of the conservation laws, the matrix of first minors of the Jacobian $\mathcal{J}=\left.\frac{\partial(\widetilde{x},\widetilde{y},\widetilde{z})}{\partial(x,y,z)}\right|_{\textrm{frame}}$. Thus,
\begin{equation*}
\mathsf{M}_{\mathcal{J}}=\left(\begin{array}{ccc}
\dfrac{v_yw_z-v_zw_y}{|B|} & \dfrac{v_xw_z-v_zw_x}{|B|} & \dfrac{v_xw_y-v_yw_x}{|B|}\\
\dfrac{u_yw_z-u_zw_y}{|B|} & \dfrac{u_xw_z-u_zw_x}{|B|} & \dfrac{u_xw_y-u_yw_x}{|B|}\\
u_yv_z-u_zv_y & u_xv_z-u_zv_x & u_xv_y-u_yv_x\end{array}\right).
\end{equation*}
Hence, the conservation laws are
$$\mathrm{d}\left(\mathcal{A}d(\rho)^{-1}\left(\boldsymbol{\upsilon}_1(J,I),\boldsymbol{\upsilon}_2(J,I),\boldsymbol{\upsilon}_3(J,I)\right)\mathsf{M}_{\mathcal{J}}\mathrm{d}^2\widehat{\mathbf{x}}\right)=0.$$

\section{Conclusion}
In Theorem 3 of \cite{GoncalvesMansfield}, it was shown that for Lagrangians which are invariant under a certain group action, and whose independent variables are left unchanged by that action, the conservation laws can be written as the product of the adjoint representation of a moving frame for the Lie group action and vectors of invariants; in this new format, the laws are handled and analysed more easily.

In this paper we have generalised this result to include cases where the independent variables of a Lagrangian participate in the action. The structure of these conservation laws differs from the ones in Theorem 3 of \cite{GoncalvesMansfield} by a matrix factor, which represents the action on the $(p-1)$-forms, and by some invariant terms in the vectors of invariants, $\boldsymbol{\upsilon}_i(J,I)$. 

It is interesting to note that from (\ref{simplofvolform}) we know that 
$$\mathrm{d}\left(\mathcal{A}d\rho^{-1}(\boldsymbol{\upsilon}_1,...,\boldsymbol{\upsilon}_p)\mathsf{M}_{\mathcal{J}}\mathrm{d}^{p-1}\widehat{\mathbf{x}}\right)=0$$
is equivalent to
$$\sum_{i=1}^p\mathcal{D}_i\left(\mathcal{A}d(\rho)^{-1}\boldsymbol{\upsilon}_i I(\mathrm{d}^p\mathbf{x})\right)=0,$$
which simplifies to an equivalent form of the Euler-Lagrange system,
$$\sum_{i=1}^p\left(\mathcal{D}_i(\boldsymbol{\upsilon}_i)-\mathcal{D}_i(\mathcal{A}d(\rho))\mathcal{A}d(\rho)^{-1}\boldsymbol{\upsilon}_i+c_i(J,I)\boldsymbol{\upsilon}_i\right)=0,$$
where $\mathcal{D}_i(\mathcal{A}d(\rho))\mathcal{A}d(\rho)^{-1}$ is known as the \textit{curvature matrix} (see \cite{MansfieldvanderKamp} for further details), which is invariant, and $c_i(J,I)$ is the coefficient of $I(\mathrm{d}^p\mathbf{x})$ in $\mathcal{D}_i(I(\mathrm{d}^p\mathbf{x}))$.

This new version of Noether's conservation laws brings insight into the structure of the laws. Using invariants and a frame usually condenses the number of terms needed to write down the laws, and makes explicit their structure by using the same invariants as those needed to write down the Euler-Lagrange equations. As for Theorem 3 in \cite{GoncalvesMansfield}, Theorem \ref{conslawsnoninvindvar} can simplify finding the solution for the extremals for one-dimensional invariant variational problems, as was shown in our motivating example, provided the Adjoint representation is non trivial.

\appendix
\section{Appendix}\label{apendice}
In this appendix, we give the proof of Lemma \ref{lemmaaboutcoeff} which shows how an element $g \in G$ acts on a differential form. Furthermore, we present a result on the Adjoint action as induced on the generating infinitesimal vector fields, which is equivalent to Theorem 3.3.10 of \cite{Mansfield}, but whose format allows to prove Theorem \ref{conslawsnoninvindvar}.

\begin{proof} \textbf{of Lemma \ref{lemmaaboutcoeff}} 
We have
$$
\mathrm{d}\widetilde{x_j}\wedge (-1)^{k-1}\mathrm{d}\widetilde{x_1}...\,\widehat{\mathrm{d}\widetilde{x_k}}...\,\mathrm{d}\widetilde{x_p}
=\left\{\begin{array}{ll}\mathrm{d}\widetilde{x_1}...\,\mathrm{d}\widetilde{x_p}=\det\left(\dfrac{\mathrm{d}\widetilde{\mathbf{x}}}{\mathrm{d}\mathbf{x}}\right)\mathrm{d}x_1...\,\mathrm{d}x_p,
& j=k,\\
0,&\mbox{else.}\end{array}\right.
$$

Note that we can write
$$(-1)^{k-1}\mathrm{d}\widetilde{x_1}...\,\widehat{\mathrm{d}\widetilde{x_k}}...\,\mathrm{d}\widetilde{x_p}$$
as
$$\sum_{\ell=1}^p(-1)^{k+\ell-2}Z^k_\ell\mathrm{d}x_1...\,\widehat{\mathrm{d}x_\ell}...\,\mathrm{d}x_p$$
and therefore,
$$\mathrm{d}\widetilde{x_j}\wedge (-1)^{k-1}\mathrm{d}\widetilde{x_1}...\,\widehat{\mathrm{d}\widetilde{x_k}}...\,\mathrm{d}\widetilde{x_p}=\sum_{\ell=1}^p\dfrac{\mathrm{d}\widetilde{x_j}}{\mathrm{d}x_\ell}(-1)^{k-1}Z^k_\ell\mathrm{d}x_1...\,\mathrm{d}x_p=\delta_{jk}\det\left(\dfrac{\mathrm{d}\widetilde{\mathbf{x}}}{\mathrm{d}\mathbf{x}}\right)\mathrm{d}x_1...\,\mathrm{d}x_p,$$
i.e.
\begin{equation}\label{coeffZ}
\sum_{\ell=1}^p\dfrac{\mathrm{d}\widetilde{x_j}}{\mathrm{d}x_\ell}(-1)^{k-1}Z^k_\ell=\delta_{jk}\det\left(\dfrac{\mathrm{d}\widetilde{\mathbf{x}}}{\mathrm{d}\mathbf{x}}\right).
\end{equation}
Now (\ref{coeffZ}) implies that 
$$(-1)^{k-1}Z^k_\ell=\left(\left(\dfrac{\mathrm{d}\widetilde{\mathbf{x}}}{\mathrm{d}\mathbf{x}}\right)^{-1}\right)_{\ell k}\det\left(\dfrac{\mathrm{d}\widetilde{\mathbf{x}}}{\mathrm{d}\mathbf{x}}\right),$$
as $\left(\dfrac{\mathrm{d}\widetilde{\mathbf{x}}}{\mathrm{d}\mathbf{x}}\right)^{-1}\dfrac{\mathrm{d}\widetilde{\mathbf{x}}}{\mathrm{d}\mathbf{x}}=\dfrac{\mathrm{d}\widetilde{\mathbf{x}}}{\mathrm{d}\mathbf{x}}\left(\dfrac{\mathrm{d}\widetilde{\mathbf{x}}}{\mathrm{d}\mathbf{x}}\right)^{-1}=I$. $\hfill \Box$
\end{proof}

\begin{theorem}\label{equivalentto3.3.10}
Let $(a_1,...,a_r)$ be coordinates on the Lie group $G$ and let the infinitesimal vector field with respect to the coordinate $a_j$ be given as 
$$\mathbf{v}_j=\Xi_jD_\mathbf{x}+\mathscr{Q}_j\boldsymbol{\nabla_{u^\alpha_J}},$$
where $\Xi_j=(\xi^1_j,...,\xi^p_j)$, $\mathscr{Q}_j=(Q^1_j,...,Q^q_j,D_1Q^1_j,...)$, $D_\mathbf{x}=(D_1,...,D_p)$ and $\boldsymbol{\nabla_{u^\alpha_J}}=(\partial_{u^1},...,\partial_{u^q},\partial_{u^1_1},...)$.  Let $\mathcal{A}d(g)$ be the Adjoint representation of $G$ with respect to the $\mathbf{v}_j$. Then the action of $g\in G$ on $\mathbf{v}_j$ is
\begin{align}\nonumber
&g\cdot \left(\left(\begin{array}{cc}\Xi_j(\zede) & \mathscr{Q}_j(\zede)\end{array}\right)\left(\begin{array}{c}D_\mathbf{x}\\
\boldsymbol{\nabla_{u^\alpha_J}}\end{array}\right)\right)=\left(\begin{array}{cc}\Xi_j(\widetilde{\zede}) & \mathscr{Q}_j(\widetilde{\zede})\end{array}\right)\\[10pt]
&\;\scriptstyle{\times} \left(\begin{array}{cc}\left(\dfrac{\mathrm{d}\widetilde{\mathbf{x}}}{\mathrm{d}\mathbf{x}}\right)^{-T} & \; \mathrm{O}\\[10pt]\label{adjointactionnew}
-\left(\dfrac{\partial \widetilde{u^\alpha_J}}{\partial u^\alpha_J}\right)^{-T}\left(\dfrac{\partial\widetilde{\mathbf{x}}}{\partial u^\alpha_J}\right)^T \mathcal{X}^{-1} & \;\left(\dfrac{\partial \widetilde{u^\alpha_J}}{\partial u^\alpha_J}\right)^{-T}\left(\dfrac{\partial\widetilde{\mathbf{x}}}{\partial u^\alpha_J}\right)^T \mathcal{X}^{-1}\left(\dfrac{\mathrm{d}u^\alpha_J}{\mathrm{d}\mathbf{x}}\right)^T+\mathcal{Y}^{-1}
\end{array}\right)\left(\begin{array}{c}D_\mathbf{x}\\
\boldsymbol{\nabla_{u^\alpha_J}} \end{array}\right),
\end{align}
where 
$$\begin{array}{rl}
\mathcal{X}=&\kern -8pt \left(\dfrac{\partial \widetilde{\mathbf{x}}}{\partial \mathbf{x}}\right)^T-\left(\dfrac{\partial \widetilde{u^\alpha_J}}{\partial \mathbf{x}}\right)^T\left(\dfrac{\partial \widetilde{u^\alpha_J}}{\partial u^\alpha_J}\right)^{-T}\left(\dfrac{\partial\widetilde{\mathbf{x}}}{\partial u^\alpha_J}\right)^T,\\\nonumber
\mathcal{Y}=&\kern-8pt \left(\dfrac{\partial \widetilde{u^\alpha_J}}{\partial u^\alpha_J}\right)^T-\left(\dfrac{\partial\widetilde{\mathbf{x}}}{\partial u^\alpha_J}\right)^T\left(\dfrac{\partial \widetilde{\mathbf{x}}}{\partial \mathbf{x}}\right)^{-T}\left(\dfrac{\partial \widetilde{u^\alpha_J}}{\partial \mathbf{x}}\right)^T,\\
\mathrm{O}=&\kern-8pt \textrm{ zero matrix.}
\end{array}$$
Furthermore, 
\begin{equation}\label{actionXiDx}
\mathcal{A}d(g)\Xi(\zede)=\Xi(\widetilde{\zede})\left(\dfrac{\mathrm{d}\widetilde{\mathbf{x}}}{\mathrm{d}\mathbf{x}}\right)^{-T}-\mathscr{Q}(\widetilde{\zede})\left(\dfrac{\partial \widetilde{u^\alpha_J}}{\partial u^\alpha_J}\right)^{-T}\left(\dfrac{\partial\widetilde{\mathbf{x}}}{\partial u^\alpha_J}\right)^T \mathcal{X}^{-1}
\end{equation}
and
\begin{equation}\label{actionQDu}
\mathcal{A}d(g)\mathscr{Q}(\zede)=\mathscr{Q}(\widetilde{\zede})\left(\left(\dfrac{\partial \widetilde{u^\alpha_J}}{\partial u^\alpha_J}\right)^{-T}\left(\dfrac{\partial\widetilde{\mathbf{x}}}{\partial u^\alpha_J}\right)^T \mathcal{X}^{-1}\left(\dfrac{\mathrm{d}u^\alpha_J}{\mathrm{d}\mathbf{x}}\right)^T+\mathcal{Y}^{-1}\right).
\end{equation}
\end{theorem}

\begin{proof}
We know that
\begin{equation}\label{actionDx}
g\cdot \left(\begin{array}{c}\boldsymbol{\nabla_\mathbf{x}}\\
\boldsymbol{\nabla_{u^\alpha_J}}\end{array}\right)=\left(\begin{array}{cc}\dfrac{\partial \widetilde{\mathbf{x}}}{\partial \mathbf{x}}& \dfrac{\partial \widetilde{\mathbf{x}}}{\partial u^\alpha_J}\\[10pt]
\dfrac{\partial \widetilde{u^\alpha_J}}{\partial \mathbf{x}} & \dfrac{\partial \widetilde{u^\alpha_J}}{\partial u^\alpha_J}
\end{array}\right)^{-T}\left(\begin{array}{c}\boldsymbol{\nabla_\mathbf{x}}\\ \boldsymbol{\nabla_{u^\alpha_J}}\end{array}\right),
\end{equation}
where
$$\left(\begin{array}{cc}\dfrac{\partial \widetilde{\mathbf{x}}}{\partial \mathbf{x}}& \dfrac{\partial \widetilde{\mathbf{x}}}{\partial u^\alpha_J}\\[10pt]
\dfrac{\partial \widetilde{u^\alpha_J}}{\partial \mathbf{x}} & \dfrac{\partial \widetilde{u^\alpha_J}}{\partial u^\alpha_J}
\end{array}\right)^{-T}=\left(\begin{array}{cc} \mathcal{X}^{-1} & \;-\left(\dfrac{\partial \widetilde{\mathbf{x}}}{\partial \mathbf{x}}\right)^{-T}\left(\dfrac{\partial \widetilde{u^\alpha_J}}{\partial \mathbf{x}}\right)^T\mathcal{Y}^{-1}\\[10pt]
-\left(\dfrac{\partial\widetilde{u^\alpha_J}}{\partial u^\alpha_J}\right)^{-T}\left(\dfrac{\partial \widetilde{\mathbf{x}}}{\partial u^\alpha_J}\right)^T\mathcal{X}^{-1} & \mathcal{Y}^{-1}\end{array}\right),$$
which was calculated using a result in \cite{HendersonSearle,Hotelling} since we assume $\partial \widetilde{\mathbf{x}}/\partial \mathbf{x}$ and $\partial \widetilde{u^\alpha_J}/\partial u^\alpha_J$ are nonsingular.

Letting $g\in G$ act on $D_\mathbf{x}$, we obtain
\begin{align}\nonumber
&g \cdot D_\mathbf{x}=\boldsymbol{\nabla_{\widetilde{\mathbf{x}}}}+\left(\dfrac{\mathrm{d} \widetilde{u^\alpha_J}}{\mathrm{d}\widetilde{\mathbf{x}}}\right)^T\boldsymbol{\nabla_{\widetilde{u^\alpha_J}}}=\mathcal{X}^{-1}\boldsymbol{\nabla_{\mathbf{x}}}-\left(\dfrac{\partial \widetilde{\mathbf{x}}}{\partial\mathbf{x}}\right)^{-T}\left(\dfrac{\partial \widetilde{u^\alpha_J}}{\partial \mathbf{x}}\right)^T\mathcal{Y}^{-1}\boldsymbol{\nabla_{u^\alpha_J}}\\\nonumber
&\quad+\left(\dfrac{\mathrm{d}\widetilde{\mathbf{x}}}{\mathrm{d}\mathbf{x}}\right)^{-T}\left(\dfrac{\mathrm{d}\widetilde{u^\alpha_J}}{\mathrm{d}\mathbf{x}}\right)^T\left[-\left(\dfrac{\partial \widetilde{u^\alpha_J}}{\partial u^\alpha_J}\right)^{-T}\left(\dfrac{\partial \widetilde{\mathbf{x}}}{\partial u^\alpha_J}\right)^T\mathcal{X}^{-1}\boldsymbol{\nabla_{\mathbf{x}}}+\mathcal{Y}^{-1}\boldsymbol{\nabla_{u^\alpha_J}}\right]\\\nonumber
&\;=\left(\dfrac{\mathrm{d}\widetilde{\mathbf{x}}}{\mathrm{d}\mathbf{x}}\right)^{-T}\left(\left(\left(\dfrac{\partial \widetilde{\mathbf{x}}}{\partial \mathbf{x}}\right)^T+\left(\dfrac{\mathrm{d} u^\alpha_\mathrm{J}}{\mathrm{d}\mathbf{x}}\right)^T\left(\dfrac{\partial \widetilde{\mathbf{x}}}{\partial u^\alpha_\mathrm{J}}\right)^T\right)\right. \\\nonumber
&\quad\left.-\left(\left(\dfrac{\partial\widetilde{u^\alpha_\mathrm{J}}}{\partial \mathbf{x}}\right)^T+\left(\dfrac{\mathrm{d} u^\alpha_\mathrm{J}}{\mathrm{d}\mathbf{x}}\right)^T\left(\dfrac{\partial \widetilde{u^\alpha_\mathrm{J}}}{\partial u^\alpha_\mathrm{J}}\right)^T \right)\left(\dfrac{\partial \widetilde{u^\alpha_\mathrm{J}}}{\partial u^\alpha_\mathrm{J}}\right)^{-T}\left(\dfrac{\partial\widetilde{\mathbf{x}}}{\partial u^\alpha_\mathrm{J}}\right)^T\right)\mathcal{X}^{-1}\boldsymbol{\nabla_{\mathbf{x}}}\\\nonumber
&\quad +\left(\dfrac{\mathrm{d}\widetilde{\mathbf{x}}}{\mathrm{d}\mathbf{x}}\right)^{-T}\left(\left(\left(\dfrac{\partial\widetilde{u^\alpha_\mathrm{J}}}{\partial \mathbf{x}}\right)^T+\left(\dfrac{\mathrm{d} u^\alpha_\mathrm{J}}{\mathrm{d}\mathbf{x}}\right)^T\left(\dfrac{\partial \widetilde{u^\alpha_\mathrm{J}}}{\partial u^\alpha_\mathrm{J}}\right)^T \right)\right.\\\nonumber
&\quad\left.-\left(\left(\dfrac{\partial \widetilde{\mathbf{x}}}{\partial \mathbf{x}}\right)^T+\left(\dfrac{\mathrm{d} u^\alpha_\mathrm{J}}{\mathrm{d}\mathbf{x}}\right)^T\left(\dfrac{\partial \widetilde{\mathbf{x}}}{\partial u^\alpha_\mathrm{J}}\right)^T\right)\left(\dfrac{\partial\widetilde{\mathbf{x}}}{\partial \mathbf{x}}\right)^{-T}\left(\dfrac{\partial \widetilde{u^\alpha_\mathrm{J}}}{\partial \mathbf{x}}\right)^T\right)\mathcal{Y}^{-1}\boldsymbol{\nabla_{u^\alpha_J}}
\\\nonumber
&\;=\left(\dfrac{\mathrm{d}\widetilde{\mathbf{x}}}{\mathrm{d}\mathbf{x}}\right)^{-T}\left(\mathcal{X}\mathcal{X}^{-1}\boldsymbol{\nabla_{\mathbf{x}}}+\left(\dfrac{\mathrm{d}u^\alpha_\mathrm{J}}{\mathrm{d}\mathbf{x}}\right)^T\mathcal{Y}\mathcal{Y}^{-1}\boldsymbol{\nabla_{u^\alpha_J}}\right)\\\nonumber
&\;=\left(\dfrac{\mathrm{d}\widetilde{\mathbf{x}}}{\mathrm{d}\mathbf{x}}\right)^{-T}D_\mathbf{x}.
\end{align}
Note that we have used $D_\mathbf{x}=\boldsymbol{\nabla_\mathbf{x}}+(\mathrm{d}u^\alpha_J/\mathrm{d}\mathbf{x})^T\boldsymbol{\nabla_{u^\alpha_J}}$ and the chain rule. 

From (\ref{actionDx}) we already know what the action of $g\in G$ is on $\boldsymbol{\nabla_{u^\alpha_J}}$; we just need to substitute $\boldsymbol{\nabla_{\mathbf{x}}}$ by $D_\mathbf{x}-(\mathrm{d} u^\alpha_J/\mathrm{d}\mathbf{x})^T\boldsymbol{\nabla_{u^\alpha_J}}$ to obtain
\begin{align}\nonumber
&g\cdot \boldsymbol{\nabla_{u^\alpha_J}}=-\left(\dfrac{\partial \widetilde{u^\alpha_J}}{\partial u^\alpha_J}\right)^{-T}\left(\dfrac{\partial \widetilde{\mathbf{x}}}{\partial u^\alpha_J}\right)^T\mathcal{X}^{-1}D_\mathbf{x}\\\nonumber
&\quad+\left[\left(\dfrac{\partial \widetilde{u^\alpha_J}}{\partial u^\alpha_J}\right)^{-T}\left(\dfrac{\partial \widetilde{\mathbf{x}}}{\partial u^\alpha_J}\right)^T\mathcal{X}^{-1}\left(\dfrac{\mathrm{d} u^\alpha_J}{\mathrm{d}\mathbf{x}}\right)^T+\mathcal{Y}^{-1}\right]\boldsymbol{\nabla_{u^\alpha_J}}.
\end{align}
This completes the proof of (\ref{adjointactionnew}).

Since $\mathbf{v}_j=\Xi_jD_\mathbf{x}+\mathscr{Q}_j\boldsymbol{\nabla_{u^\alpha_J}}$ can be written as $\Xi_j\boldsymbol{\nabla_\mathbf{x}}+\Phi_j\boldsymbol{\nabla_{u^\alpha_J}}$, by Theorem 3.3.10 in \cite{Mansfield} we know that
\begin{align}\nonumber
&\mathcal{A}d(g)\left(\begin{array}{cc} \Xi(\zede) & \mathscr{Q}(\zede)\end{array}\right)\left(\begin{array}{c} D_\mathbf{x}\\\boldsymbol{\nabla_{u^\alpha_J}}\end{array}\right)=\left(\begin{array}{cc} \Xi(\widetilde{\zede}) & \mathscr{Q}(\widetilde{\zede})\end{array}\right)\\[10pt]\nonumber
&\;\scriptstyle{\times} \left(\begin{array}{cc}\left(\dfrac{\mathrm{d}\widetilde{\mathbf{x}}}{\mathrm{d}\mathbf{x}}\right)^{-T} & \; \mathrm{O}\\[10pt]
-\left(\dfrac{\partial \widetilde{u^\alpha_J}}{\partial u^\alpha_J}\right)^{-T}\left(\dfrac{\partial\widetilde{\mathbf{x}}}{\partial u^\alpha_J}\right)^T \mathcal{X}^{-1} & \;\left(\dfrac{\partial \widetilde{u^\alpha_J}}{\partial u^\alpha_J}\right)^{-T}\left(\dfrac{\partial\widetilde{\mathbf{x}}}{\partial u^\alpha_J}\right)^T \mathcal{X}^{-1}\left(\dfrac{\mathrm{d}u^\alpha_J}{\mathrm{d}\mathbf{x}}\right)^T+\mathcal{Y}^{-1}
\end{array}\right)\left(\begin{array}{c}D_\mathbf{x}\\
\boldsymbol{\nabla_{u^\alpha_J}} \end{array}\right);
\end{align}
from this we can easily read the results (\ref{actionXiDx}) and (\ref{actionQDu}). $\hfill \Box$
\end{proof}

\begin{center}
\textsc{Universidade Federal de S\~{a}o Carlos\\University of Kent}
\end{center}
\end{document}